\documentclass[1p,final]{elsarticle}
\usepackage{amsfonts,color,float,morefloats,array}
\usepackage{pslatex,graphicx,pgfplots,stmaryrd}
\usepackage{amssymb,amsthm,mathrsfs,amsmath,amstext,latexsym,mathtools}

\usepackage{mathcomp}
\usepackage{tikz}
\usetikzlibrary{matrix,arrows}

\allowdisplaybreaks[4]

\newtheorem{thm}{Theorem}[section]
\newtheorem{cor}[thm]{Corollary}
\newtheorem{lem}[thm]{Lemma}
\newtheorem{prop}[thm]{Proposition}
\newtheorem{rem}[thm]{Remark}

\newtheorem{defn}[thm]{Definition}

\newtheorem{example}[thm]{Example}

\newtheorem*{thmletter}{Theorem 0}

\numberwithin{equation}{section}

\newcommand{\ceil}[1]{\left\lceil {#1} \right\rceil}
\newcommand{\floor}[1]{\left\lfloor {#1} \right\rfloor}

\newcommand{\Span}{\operatorname{span}}  

\newcommand{\Div}{\operatorname{div}}

\begin{document}

\begin{frontmatter}



\title{Weierstrass Semigroups From a Tower of Function Fields 
	Attaining the Drinfeld-Vl{\u{a}}du{\c{t}} Bound
	\tnoteref{mytitlenote}
}
\tnotetext[mytitlenote]{This work is partly supported by the NSFC (11701317, 11531007, 11801303) and Tsinghua University Start-up fund. This work is also partly supported by Guangzhou Science and Technology Program (201607010144) and the Natural Science Foundation of Shandong Province of China (ZR2016AM04, ZR2019QA016).}

\author[mymainadress]{Shudi Yang}
\ead{yangshudi@qfnu.edu.cn}
\address[mymainadress]{School of Mathematical	Sciences, Qufu Normal University, Shandong 273165, P R China }

\author[mysecondaryaddress]{Chuangqiang Hu\corref{mycorrespondingauthor}}
\cortext[mycorrespondingauthor]{Corresponding author}
\ead{huchq@tsinghua.edu.cn}

\address[mysecondaryaddress]{Yau Mathematical Sciences Center, Tsinghua University, Peking 100084, P R China}
\begin{abstract}
For applications in algebraic geometric codes, an explicit description of bases of Riemann-Roch spaces
of divisors on function fields over finite fields is needed. We investigate the third function field $ F^{(3)} $ in a tower of Artin-Schreier extensions  described by Garcia and
Stichtenoth reaching the Drinfeld-Vl{\u{a}}du{\c{t}}
bound. We construct bases for the related Riemann-Roch spaces on $ F^{(3)} $ and present some basic properties of divisors on a line. From the bases, we
explicitly calculate the Weierstrass semigroups and pure gaps at several places on $  F^{(3)} $. All of these results can be viewed as a generalization of the previous work done by Voss and H\o{}holdt (1997).
\end{abstract}


\end{frontmatter}

\section{Introduction}
Goppa constructed error-correcting linear codes by using tools from algebraic geometry (AG): a non-singular, projective, geometrically irreducible, algebraic curve $\mathcal{X}$ of genus $g$ defined over $\mathbb{F}_{q}$, the finite field with $q$ elements, and two rational divisors $D$ and $G$ on $\mathcal{X}$. These divisors are chosen in such a way that they have disjoint supports, and $D$ is equal to a sum of pairwise distinct rational places, $D=P_{1}+P_2+\cdots+P_{n}$. The algebraic geometric code is defined as
\[C_{\mathcal{L }}( D,G):=\Big\{\, (f(P_{1}), f(P_{2}),\cdots, f(P_{n})) \,\Big|\,f\in\mathcal{L}(G)\,\Big\} ,\]
where $\mathcal{L}(G)$ denotes the Riemann-Roch space associated to $ G $, see \cite{stichtenoth2009algebraic} as general references for all facts concerning algebraic geometric codes.

There are several problems arising from the construction of AG codes with good parameters.
\begin{itemize}
	\item Choose a function field $F$ with many rational places and small genus.
	\item Describe the places (especially the rational places) of $F$ explicitly.
	\item Determine a basis for the Riemann-Roch space of a given divisor. 
	\item Determine the minimal distance of AG codes.
\end{itemize}

In 1995 and 1996, Garcia and Stichtenoth \cite{Garcia1995,Garcia1996O} showed that two families of explicitly described curves also achieve the Drinfeld-Vl{\u{a}}du{\c{t}} bound $ A(q^2) = q-1 $. These curves are presented in the language of function
fields, specifically as towers of extensions of the rational function field $ K(x_1) $, where $  K:= \mathbb{F}_{q^2} $. One example they gave in \cite{Garcia1996O} is the tower 
$ \mathcal{T}=(\mathcal{T}_1 , \mathcal{T}_2 ,  \cdots,\mathcal{T}_n ) $ that has description as the following sequence of Artin-Schreier extensions:
\[\mathcal{T}_1 = K(x_1),\]
and for $ 1 \leqslant k \leqslant n-1 $, 
\[\mathcal{T}_{k+1} = \mathcal{T}_{k}(x_{k+1}),\]
where \begin{align*}
x_{k+1}^q+x_{k+1}=\dfrac{x_k^q}{x_k^{q-1}+1 }. 
\end{align*}
Identifying the generator matrix for one-point AG codes constructed
on this tower requires the determination of a basis for the
vector spaces $\{\mathcal{L}(rP)|0\leqslant r \in \mathbb{Z}\}$, which comprise functions having
poles only at a specified point $ P $. Let $ P_{\infty}^{(1)} $ denote the unique pole of $ x_1 $ in $ \mathcal{T}_1 $ and $ P_{\infty}^{(n)} $
the unique place in $ \mathcal{T}_n $ lying above  $ P_{\infty}^{(1)} $. In \cite{Aleshnikov1999}, the authors constructed basis for these spaces, using techniques borrowed from algebraic number theory. Later, Aleshnikov {\emph{et al.}} \cite{Aleshnikov2001On} gave a description of how places split in this asymptotically optimal tower of function fields and they provided
an exact count of the number of places of degree one.

Our interest here is in the second tower of Artin-Schreier extensions over $ K $. According to \cite{Garcia1995}, the AG codes derived from these function fields have better coefficients in excess of Weil bound. We introduce the function fields of this tower in the following way. 
\begin{defn}[\cite{Garcia1995}, Definition 0.1]\label{def:tow}
	Let $ F^{(1)} := K(x_1) $ be the rational function field over $ K $. For $ i \geqslant 1 $ let 
	\[ 	F^{(i+1)} := F^{(i)} (z_{i+1}) , \]
	where $ z _ {i+1}  $ satisfies the equation
	\[ z_{i+1}^q + z_{i+1} =x_i^{q+1 } , \]
	with 
	\[   	x_{i+1} := \frac{z_{i+1} }{x_{i}} \in F^{(i+1)} . \]
\end{defn}

Another construction of a recursive tower is due to Bassa, Beelen, Garcia and Stichtenoth \cite{Bassa2015}. Using this tower they obtained an improvement of the Gilbert-Varshamov (GV) bound for all non-prime fields $ \mathbb{F}_l $ with $ l \leqslant 49 $, except possibly $ l = 125 $ in \cite{Bassa2014}. The curve associated to the second function fields of this tower is called BBGS curve according to \cite{Bassa2015}. Recently, Hu \cite{Hu2017} presented multi-point AG codes overstepping the GV bound from BBGS curves in \cite{Bassa2015}.

In the work \cite{Garcia1995}, the authors pointed out all the rational places on such function fields. However, it is really complicated to compute the bases of Riemann-Roch spaces related to these function fields.
The second function field $  {F}^{(2)} $ is the Hermitian function field, which is known to be the unique function field over $ K $ of genus $ g = q(q - 1)/2  $ that attains the Weil bound \cite{Ruck1994}. The basic fact about Hermitian function field is that there exists only one rational place $P_{\infty} $ on the infinity. A basis of the Riemann-Roch space attached to this place can be given by the monomial polynomials, i.e.,     
\[ \mathcal{L}(rP_{\infty} ) = \Span \Big\{\, x_1^i z_2^j \,\Big|\, 0 \leqslant i,\,\,  0 \leqslant j\leqslant q-1, \,\, iq+j(q+1)\leqslant r \,\Big\}. \]   
See Lemma 6.4.4 in \cite{stichtenoth2009algebraic} for more information. Now we come to the third function field $ F^{(3)} $. Let us introduce two rational places  $ P^{(3)} $ and $ Q^{(3)} $ and two divisors $ S_0^{(3)} $ and $ S_1^{(3)} $. They can be expressed formally as  
\begin{enumerate}
	\item $ P^{(3)} = (x_1 = \infty )$,
	
	\item $ Q^{(3)} = (x_1 = x_2 =x_3 =0 )$,
	
	\item $ S_0^{(3)} = (x_1 = x_2 =0, x_3 \not =0 )$, and
	
	\item  $ S_1^{(3)} = (x_1 = 0, x_2 \not =0 )$.
\end{enumerate}
Actually, the divisor $  S_0^{(3)} $ (resp. $  S_1^{(3)} $) contains $ q-1 $ places in $ F^{(3)} $ denoted by $  S_{0,\mu}^{(3)} $  (resp. $  S_{1,\mu}^{(3)} $) for $ 1\leqslant \mu \leqslant q-1 $. This means that $  S_{0,\mu}^{(3)} $  (resp. $  S_{1,\mu}^{(3)} $) is the unique zero of $  z_3-\alpha_\mu $  (resp. $  z_2-\alpha_\mu $), where $\alpha_\mu \in \mathbb{F}_{q^2}^*  $ and $ \alpha_\mu^{q-1}+1=0 $. In other words, we have the following decompositions used when we consider divisors on a line:
\[S_0^{(3)}= \sum_{\mu=1}^{q-1}  S_{0,\mu}^{(3)}, \quad  S_1^{(3)}= \sum_{\mu=1}^{q-1}  S_{1,\mu}^{(3)} .\] 

In general, it is a hard problem to determine a basis for a Riemman-Roch space. However, Voss and H\o{}holdt \cite{Voss1997} gave a basis of the Riemann-Roch space $ \mathcal{L}( t S_1^{(3)} + u P^{(3)} ) $.
\begin{prop} [\cite{Voss1997}, Theorem 4.3]\label{prop:basi1}
	The set 
	\[ B(t,u): = \Big\{ \,x_1 ^{i_1 } x_2 ^{i_2 } z_3 ^j\, \Big|\, (i_1, i_2, j)\in I(t,u) \, \Big\} \]
	is a basis of $ \mathcal{L}( t S_1^{(3)} + u P^{(3)} ) $ over $ K $, where
	$I(t,u):= I_1(t,u)\cup I_2(t,u)\cup I_3(t,u) $ with
	\begin{align*}
	I_1(t,u) &= \Big\{(i_1,i_2,j) \,|\,0\leqslant i_1,\, 0\leqslant i_2,j\leqslant q\!-\!1, 
	i_2 q\!+\!j(q+1) \leqslant u-i_1 q^2,\,
	i_2 q+j(q+1) \leqslant t+i_1 q \Big\},\\
	I_2(t,u) &= \Big\{(i_1,-i_2,j) \,|\,1\leqslant i_2\leqslant q,\, i_1 \geqslant i_2 q,\, 0\leqslant j\leqslant q-1,\,
	j(q+1) \leqslant u + i_2 q - i_1 q^2,\,\\
	&\phantom{= \{(i_1,-i_2,j)  \}}
	j(q+1) \leqslant t + i_2 q + i_1 q,\,
	(i_1-1,-i_2+q,j) \notin I_1(t,u) \Big\}  ,\\
	I_3(t,u) &= \Big\{(-i_1,i_2,j) \,|\,1\leqslant i_1,\,1\leqslant i_2 \geqslant   q-1,\, i_1 \geqslant i_2 q,\, 0\leqslant j\leqslant q-1,\\
	&\phantom{= \{(i_1,-i_2,j) \}} i_2 q+j(q+1) \leqslant u+i_1 q^2,\,
	i_2 q+j(q+1) \leqslant t-i_1 q \Big\}  . 
	\end{align*}
\end{prop}
 Voss and H\o{}holdt \cite{Voss1997} succeeded in counting the elements in the set $ B(t,u) $, see Proposition 4.9 in \cite{Voss1997}. The set $ I(t,u) $ above is constructed explicitly, but it seems complicated and unnatural.
  So this work is intended as an attempt to motivate the result of \cite{Voss1997}. Precisely speaking, we would give another natural and brief exposition of the bases for various Riemann-Roch spaces $ \mathcal{L}(rQ^{(3)} + s S_0^{(3)}+t S_1^{(3)} + u P^{(3)}  ) $ in $ F^{(3)} $ and $\mathcal{L}( r Q^{(3)}+ \sum_{\mu=1}^{q-1} s_{\mu}  S_{0,\mu}^{(3)}  + \sum_{\nu=1}^{q-1}  t_{\nu} S_{1,\nu}^{(3)}+ u P^{(3)}) $. This is briefly described below.
  
  \begin{thmletter}\label{thm:basis}
  	Let $ w=  {x_1^{q}}/{x_2 } $ and $ v= {x_1^{q }x_2^{q-1} }/{  x_3  } $. 
  	\begin{enumerate}
  		\item The elements $ x_1^i w^j v^k $ with $ (i,j,k) $ satisfying certain conditions form a basis of the Riemann-Roch space $\mathcal{L}(r Q^{(3)} + s S_0^{(3)}+t S_1^{(3)} + u P^{(3)} )$.
  		\item The elements 
  		$ 	x_1^{i} \prod_{\nu=1}^{q-1} {(z_2 -\alpha_\nu)^ {j_\nu} } \prod_{\mu=1}^{q-1} {(z_3 -\alpha_\mu)^ {k_\mu} } $ with $ (i, j_1,\cdots,j_{q-1} , k_1,\cdots,k_{q-1} )$ satisfying certain conditions form a basis of the Riemann-Roch space $\mathcal{L}( r Q^{(3)}+ \sum_{\mu=1}^{q-1} s_{\mu}  S_{0,\mu}^{(3)}  + \sum_{\nu=1}^{q-1}  t_{\nu} S_{1,\nu}^{(3)}+ u P^{(3)}) $ .
  	\end{enumerate} 
  \end{thmletter}
  It seems that the difficult part is to find enough linearly independent elements of a special form. Fortunately we can find them by considering a natural form of elements and counting these elements. Then we utilize these bases to calculate the Weierstrass semigroups and the pure gap sets at certain places on $  F^{(3)} $. It should be noted that Weierstrass semigroups and pure gaps are of vital use in finding AG codes with good parameters. We refer the reader to the works \cite{Bartoli2017,Bartoli2018Pure,Bartoli2018Kummer,HuYang2016Kummer,Hu2017multi,matthews2001weierstrass,matthews2005weierstrass,YangHu2017,Yang2018Pure} for the technique to improve the parameters of AG codes using Weierstrass semigroups and pure gaps.

The paper is organized as follows. In Section \ref{sec:arith}, we introduce some arithmetic properties of the function fields of the second tower introduced in Definition \ref{def:tow} and construct a basis for the Riemann-Roch space $ \mathcal{L}(rQ^{(3)} + s S_0^{(3)}+t S_1^{(3)} + u P^{(3)}  ) $. Section \ref{sec:divi} is
devoted to investigating the properties of divisors on a line and finding a basis of the Riemann-Roch space $\mathcal{L}( r Q^{(3)}+ \sum_{\mu=1}^{q-1} s_{\mu}  S_{0,\mu}^{(3)}  + \sum_{\nu=1}^{q-1}  t_{\nu} S_{1,\nu}^{(3)}+ u P^{(3)}) $. Finally, we demonstrate the Weierstrass semigroups and pure gaps at certain places on $  F^{(3)} $ in Section \ref{sec:Weiersemipureg}.

\section{The arithmetic properties of the tower}\label{sec:arith}

We start with some notations. Let $ q  $ be a power of a prime $p $ and $ K= \mathbb{F}_{q^{2}} $ be a finite field of cardinality $ q^{2}  $. Let $\mathcal{F} = \Big\{F^{(1)},F^{(2)},F^{(3)},\cdots \Big\}$ be the tower of function fields as defined in Definition \ref{def:tow}. 
We denote by $ P^{(1)}$ and $Q^{(1)}$  the pole and the zero of $ x_1 $ in $ F^{(1)} $ respectively. Let $ \mathbb{P}(F^{(n)}) $ be the set of places of the function field $ F^{(n)}/K $ for $ n \geqslant 1 $. The following diagram Figure \ref{Fig:Sets of places} gives a classification of rational places over $ P^{(1)} $ and $ Q^{(1)}$ in $ \mathbb{P}( F^{(n)} )$ and distinguishes the ramifications of distinct rational places, where $ P^{(n)} $
and $ Q^{(n)} $ are rational places and the divisors $ S_k^{(n)} $ are combined with some rational places for $ n \geqslant 1 $ and $ k\geqslant 0 $. The integers in Figure \ref{Fig:Sets of places} are the ramification indices of related places, for instance, the ramification index $ e(P^{(2)}|P^{(1)} )=q $.
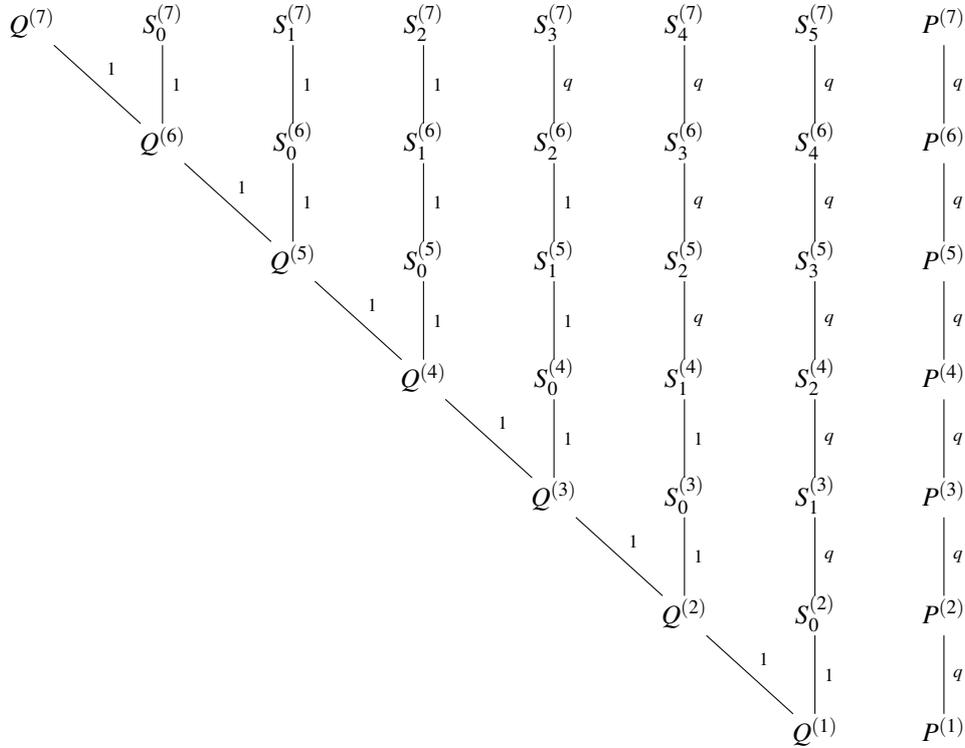
\begin{figure}[H]
	\centering
	\begin{tikzpicture}[description/.style={fill=white,inner sep=2pt}]
	\matrix (m) [matrix of math nodes, row sep=3em,
	column sep=2.5em, text height=1.5ex, text depth=0.25ex]
	{ Q^{(7)} & S_{0}^{(7)} & S_{1}^{(7)} & S_{2}^{(7)} & S_{3}^{(7)} & S_{4}^{(7)} & S_{5}^{(7)} & P^{(7)} \\ 
		&     Q^{(6)} & S_{0}^{(6)} & S_{1}^{(6)} & S_{2}^{(6)} & S_{3}^{(6)} & S_{4}^{(6)} & P^{(6)} \\  
		&             &     Q^{(5)} & S_{0}^{(5)} & S_{1}^{(5)} & S_{2}^{(5)} & S_{3}^{(5)} & P^{(5)} \\
		&             &	           &     Q^{(4)} & S_{0}^{(4)} & S_{1}^{(4)} & S_{2}^{(4)} & P^{(4)} \\  
		&	         &             &             &     Q^{(3)} & S_{0}^{(3)} & S_{1}^{(3)} & P^{(3)} \\
		&	         &             &             &             &     Q^{(2)} & S_{0}^{(2)} & P^{(2)} \\
		&             &             &             &             &             &     Q^{(1)} & P^{(1)} \\};
	\path[-,font=\scriptsize]
	(m-1-1) edge node[auto] {$ 1 $} (m-2-2)
	(m-1-2) edge node[auto] {$ 1 $} (m-2-2)
	(m-1-3) edge node[auto] {$ 1 $} (m-2-3)
	(m-1-4) edge node[auto] {$ 1 $} (m-2-4)
	(m-1-5) edge node[auto] {$ q $} (m-2-5)
	(m-1-6) edge node[auto] {$ q $} (m-2-6)
	(m-1-7) edge node[auto] {$ q $} (m-2-7)
	(m-1-8) edge node[auto] {$ q $} (m-2-8)
	(m-2-2) edge node[auto] {$ 1 $} (m-3-3)
	(m-2-3) edge node[auto] {$ 1 $} (m-3-3)
	(m-2-4) edge node[auto] {$ 1 $} (m-3-4)
	(m-2-5) edge node[auto] {$ 1 $} (m-3-5)
	(m-2-6) edge node[auto] {$ q $} (m-3-6)
	(m-2-7) edge node[auto] {$ q $} (m-3-7)
	(m-2-8) edge node[auto] {$ q $} (m-3-8)
	(m-3-3) edge node[auto] {$ 1 $} (m-4-4)
	(m-3-4) edge node[auto] {$ 1 $} (m-4-4)
	(m-3-5) edge node[auto] {$ 1 $} (m-4-5)
	(m-3-6) edge node[auto] {$ q $} (m-4-6)
	(m-3-7) edge node[auto] {$ q $} (m-4-7)
	(m-3-8) edge node[auto] {$ q $} (m-4-8)
	(m-4-4) edge node[auto] {$ 1 $} (m-5-5)
	(m-4-5) edge node[auto] {$ 1 $} (m-5-5)
	(m-4-6) edge node[auto] {$ 1 $} (m-5-6)
	(m-4-7) edge node[auto] {$ q $} (m-5-7)
	(m-4-8) edge node[auto] {$ q $} (m-5-8)
	(m-5-5) edge node[auto] {$ 1 $} (m-6-6)
	(m-5-6) edge node[auto] {$ 1 $} (m-6-6)
	(m-5-7) edge node[auto] {$ q $} (m-6-7)
	(m-5-8) edge node[auto] {$ q $} (m-6-8)
	(m-6-6) edge node[auto] {$ 1 $} (m-7-7)
	(m-6-7) edge node[auto] {$ 1 $} (m-7-7)
	(m-6-8) edge node[auto] {$ q $} (m-7-8);
	\end{tikzpicture}
	\protect\caption{Ramifications for towers  $F^{(n)}/K $}
	\label{Fig:Sets of places}
\end{figure}

For the convenience of the reader we paraphrase several  results in \cite{Garcia1995}. In the next proposition, we denote by
\[\Div_n(x_n)= \sum_{P \in \mathbb{P}(F^{(n)})} v_P(x_n) P \]
the principal divisor of $ x_n $ in the function field $ F^{(n)} $, where $ v_P $ is the normalized discrete valuation associated with $ P $.
\begin{prop}[\cite{Garcia1995}, Lemma 2.9 and Theorem 2.10]\label{prop:div_deg}
	\label{prop:valuation}
	For $ n\geqslant 1 $, we have the following assertions.
	\begin{enumerate}
		\item	$ \deg (Q^{(n)})= 1 $, $ \deg(P^{(n)})=1 $ and 
		\[ 
		\deg (S_i^{(n)})=\begin{cases} 
		q^{i}(q-1) & \text{ for } 2i+3\leqslant n,\\
		q^{n-i-2}(q-1) &  \text{ for }  2i+3 > n.
		\end{cases}
		\]
		\item   The principle divisor of $ x_n $ in $ F^{(n)} $ is
		\[
		\Div_{n} (x_n) = q^{n-1} Q^{(n)}-\sum_{i=0}^{\floor{\frac{n-3}{2}}} q^{n-2-2i}S_{i}^{(n)}- \sum_{i={\floor{\frac{n-1}{2}}}}^{n-2} S_{i}^{(n)}- P^{(n)}.
		\]
		\item The genus of $ F^{(n)} $ is given by
		\[
		g^{(n)}=\begin{cases}
		q^n +q^{n-1}-q^{\frac{n+1}{2}}-2 q^{\frac{n-1}{2}}+1 & \text{
			when $ n $ is odd,}\\
		q^n +q^{n-1}- \frac{1}{2} q^{\frac{n}{2} +1}-\frac{3}{2} q^{\frac{n}{2}}-q^{\frac{n}{2}-1}+1 &\text{
			when $ n $ is even. }
		\end{cases} 
		\]
		
	\end{enumerate}
\end{prop}
With Figure \ref{Fig:Sets of places} and Proposition \ref{prop:div_deg} we get the principle divisors of $x_1$, $x_2$ and $x_3$ in $ F^{(1)} $, $ F^{(2)} $ and $ F^{(3)} $:
\begin{align*}
\Div_1(x_1)&=Q^{(1)}-P^{(1)},\\
\Div_2(x_1)&=Q^{(2)}+S_0^{(2)}-q P^{(2)},\\
\Div_3(x_1)&=Q^{(3)}+S_0^{(3)}+qS_1^{(3)}-q^2P^{(3)},\\
\Div_2(x_2)&=q Q^{(2)}-S_0^{(2)}-P^{(2)},\\
\Div_3(x_2)&=q Q^{(3)}+q S_0^{(3)}- qS_1^{(3)}-q P^{(3)} ,\\
\Div_3(x_3)&=q^2 Q^{(3)} - q S_0^{(3)}-  S_1^{(3)}- P^{(3)} .
\end{align*}

Let $ w= \dfrac{x_1^{q}}{x_2 } $ and $ v=\dfrac{x_1^{q }x_2^{q-1} }{  x_3  }   $. Then
\begin{align*}
\Div_3 (w)& =(q^2+q) S_1^{(3)}  -(q^3-q) P^{(3)},\\
\Div_3 (v)&= (q^2+q)S_0^{(3)} +(q+1)  S_1^{(3)} - (q^3+q^2-q-1) P^{(3)}.
\end{align*}
It then follows that
\begin{align}\label{eq:E}
\Div_3 (x_1^i w^j v^k ) &=  i Q^{(3)}+ \big( i+ (q^2+q)k \big) S_0^{(3)}
+ \big( q i+ (q^2+q)j+ (q +1)k \big) S_1^{(3)} \nonumber \\
& \phantom{=} + \big( -q^2 i- (q^3-q)j- (q^3+q^2-q-1)k \big) P^{(3)}.
\end{align}
For convenience, we denote by 
$\ceil{x}$ the smallest integer not less than $ x $. Then $ j = \ceil{ {\alpha}/{\beta}\,} $ means that 
\[  j \in \mathbb{Z} \quad \text{and} \quad \alpha\leqslant \beta j < \alpha + \beta .
\]
Let us define a lattice point set
\begin{align}\label{eq:O1}
\Omega_{r,s, t, u} =\Big\{\, (i, j, k )\,\Big|\, 
-r &\leqslant i, \,
-s  \leqslant i+ (q^2+q)k <-s+(q^2+q),  \nonumber\\
-t & \leqslant q i+ (q^2+q)j+ (q +1)k <-t+(q^2+q)   ,\nonumber \\
-u & \leqslant -q^2 i- (q^3-q)j- (q^3+q^2-q-1)k    \,\Big\}, 
\end{align}
or equivalently, 
\begin{align*}
\Omega _{r,s,t,u}: =\Big\{\,(i,j,k) \,\Big| \, -r & \leqslant i ,\,\, k = \ceil{\frac{-s- i}{q^2+q}}, \,\,j  = \ceil{\frac{-t -qi -(q+1)k}{q^2+q} }, \\ 
\, -u & \leqslant -q^2 i- (q^3-q)j- (q^3+q^2-q-1)k \,   \Big\} .
\end{align*}

The following proposition is useful in calculating bases of Riemman-Roch spaces in $ F^{(3)} $, whose proof is very technical and so is given later after some preparations.
\begin{prop}\label{prop:Omega_num}
	There exists a constant $R$ depending on $ s,t, u$, such that for all $ r \geqslant R $, 
	\[ \# \Omega_{ r, s , t , u } = 1-g^{(3)} +r+(q-1)s+ (q-1)t + u . \] 
\end{prop}
Our main result is given in the next proposition, which states the bases of various Riemman-Roch spaces.
\begin{prop}\label{prop:basis}
	The elements $ x_1^i w^j v^k $ with $ (i,j,k)\in \Omega_{r,s,t,u} $ form a basis of the Riemann-Roch space $\mathcal{L}(G)$, where we denote $ G := r Q^{(3)} + s S_0^{(3)}+t S_1^{(3)} + u P^{(3)}  $.
\end{prop}
\begin{proof}
	Let $ (i,j,k) \in \Omega_{r,s, t,u} $. Then one easily gets $ x_1^i w^j v^k  \in \mathcal{L}(G) $.
	We have from \eqref{eq:E} that
	$ v_{Q^{(3)}}( x_1^i w^j v^k ) = i$, which indicates that the valuation of $  x_1^i w^j v^k  $ at the rational place $Q^{(3)}$ uniquely depends on $i$. Since the lattice points in $ \Omega_{r,s, t,u} $ provide distinct values of $i$, the elements $ x_1^i w^j v^k  $
	are linearly independent of each other, with $ (i,j,k) \in \Omega_{r,s,t,u} $. In order to show that they constitute a basis for the Riemann-Roch space $\mathcal{L}( G )$, the only thing is to prove that
	\[
	\ell (G)
	= \#\Omega_{r,s,t,u},
	\]	
	where $ \ell (G) $ is the dimension of $\mathcal{L}( G )$.
	For the case of $ r $ sufficiently large, it follows from the Riemann-Roch Theorem and Propositions \ref{prop:div_deg} and \ref{prop:Omega_num} that
	\begin{align*}
	\ell (G) & = 1-g^{(3)} + \deg(G)\\
	& = 1-g^{(3)} +r+ (q-1)s + (q-1)t+u\\
	&= \#\Omega_{r,s,t,u}.
	\end{align*}
	This implies that
	$ \mathcal{L}(G) $
	is spanned by the elements $ x_1^i w^j v^k  $ with $ (i,j,k) \in \Omega_{r,s, t,u} $.
	
	For the general case, we choose $ r' > r $ large enough and set 
	$ G' :=  r' Q^{(3)} + s S_0^{(3)}+t S_1^{(3)} + u P^{(3)}  $.
	From above argument, we know that the elements $ x_1^i w^j v^k  $ with $ (i,j,k) \in \Omega_{r',s,t,u} $ span the whole space of	$ \mathcal{L}(G') $.
	Remember that
	$ \mathcal{L}(G) $ is a linear subspace of $ \mathcal{L}(G') $, which can be written as
	\begin{equation*}
	\mathcal{L}(G) = \Big\{\, f \in \mathcal{L}(G') \,\Big|\,v_{Q^{(3)}}(f)\geqslant -r\,\Big\}.
	\end{equation*}
	Thus, we choose $ f \in \mathcal{L}(G)  $ and suppose that
	\begin{equation*}
	f=\sum_{(i,j,k) \in \Omega_{r',s,t,u}} a_{i}  x_1^i w^j v^k  ,
	\end{equation*}
	since $ f \in \mathcal{L}(G') $ by definition. The valuation of $ f $ at $ Q^{(3)} $ is $ v_{Q^{(3)}}(f)=\min_{a_i\neq 0} \{ i \}$. Then the inequality $ v_{Q^{(3)}}(f)\geqslant -r $ gives that, if $ a_i \neq 0 $, then $ i \geqslant -r $. Equivalently, if $ i < -r $, then $ a_i=0 $.  From the definition of $ \Omega_{r,s,t,u} $ and $ \Omega_{r',s,t,u} $, we get that
	\begin{equation*}
	f=\sum_{(i,j,k) \in \Omega_{r,s,t,u} } a_{i}  x_1^i w^j v^k  .
	\end{equation*}
	Thus $ \mathcal{L}(G) $ is also spanned by the elements $ x_1^i w^j v^k  $ with $ (i,j,k) \in \Omega_{r,s, t,u} $ and the theorem then follows.
\end{proof}

\begin{cor}\label{cor:basis}
	The elements $ x_1^i x_2^j x_3^k $ with $ (i,j,k)\in \Omega'_{r,s,t,u} $ form a basis of the Riemann-Roch space $\mathcal{L}(G )$, where $ G =  r Q^{(3)} + s S_0^{(3)}+t S_1^{(3)} +u P^{(3)} $ and
	\begin{align*}
	\Omega_{r,s, t, u}^{\prime} =\Big\{\, (i, j, k )\,\Big|\, -r &\leqslant i+ qj +q^2 k,\\
	-s &\leqslant i+ qj -q k< -s+(q^2+q),\\
	-t &\leqslant qi-qj-k< -t+(q^2+q),\\
	-u &\leqslant -q^2 i-qj-k  \, \Big\}.
	\end{align*}
	
\end{cor}
\begin{proof}
	By definition, we have 
	\[ x_1^i w^j v^k=x_1^{i+qj+q k} x_2^{-j+(q-1) k } x_3^{-k}. \]
	Let $ I:= i+qj+q k $, $ J:=-j+(q-1) k $, and $ K:=-k $. Then, we find that
	\[ \Big\{\,x_1^I x_2^J x_3^K \,\Big|\,(I,J,K)\in \Omega_{r,s,t,u}^{\prime} \,\Big\}=\Big\{\,x_1^i w^j z^k \,\Big|\,(i,j,k)\in \Omega_{r,s,t,u}  \,\Big\} .\]
	Now the corollary follows from Proposition \ref{prop:basis}.
\end{proof}
\begin{rem}
	When $ r=0 $ and $ s=0 $, we can obtain the bases of the Riemman-Roch space $\mathcal{L}(t S_1^{(3)} +u P^{(3)} )$ by Proposition \ref{prop:basis} and Corollary \ref{cor:basis}. Note that these are new bases different from the ones in Proposition \ref{prop:basi1}. 
\end{rem}	

The proof of Proposition \ref{prop:Omega_num} will be given in several steps.
\begin{enumerate}
	\item After transformation indicated in  Lemma \ref{lem:Omega_reduction}, we may assume that $s = 0$, $ 0 \leqslant t <q+1 $ and $q^3+q^2 \leqslant u < 2(q^3+q^2)$.
	\item Let $ r ^* := q^3-q $ and $ u^*:=  q^3+q^2  $. Lemma \ref{lem:Omega_large} tells us that the difference of  $\# \Omega_{r,0,t,u} $ and   $ \# \Omega_{r^*,0,t,u^*} $ is $(r - r^*)+(u - u^*) $. So it is enough to give a formula for $ \# \Omega_{r^*,0,t,u^*} $.
	\item Lemma \ref{lem:Omega_main} is devoted to reduce  our three-dimensional lattice point set $ \Omega_{r^*,0,t,u^*} $ to several two-dimensional lattice point sets.
	\item Using Pick's theorem,  one can easily solve the related two-dimensional counting problems,
	which is done in Lemmas \ref{lem:Psi_a} and \ref{lem:Phi_b}.
\end{enumerate}

In order to prove Proposition \ref{prop:Omega_num}, we need to do some preparations. 
\begin{prop} \label{prop:div_equi}
	Let $ \alpha,\beta \in \mathbb{Z} $. Suppose that $  t - q s  = \widehat t+  (q+1)\alpha$ with $ 0 \leqslant \widehat t < q+1 $, and $ u + q^2 s -  (q+1)\alpha = (q^3+q^2) \beta + \widehat u$. Let $ \widehat r = r-s+(q^2+q )\alpha+(q^3+q^2 )\beta $. Then 
	\begin{align*}
	r Q^{(3)}+  s S_0^{(3)}+ t S_1^{(3)}+ u P^{(3)} \sim 
	\widehat r Q^{(3)}+ \widehat{t} S_1^{(3)} + \widehat{u} P^{(3)}.
	\end{align*}  
\end{prop}
\begin{proof}
	It follows  from the divisors of $ x_1 $, $ w $ and $ v $ that
	\begin{equation*}
	\Div_3 (x_1^{q^2+q} w^{-q}  v^{-1} ) = (q^2+q) Q^{(3)} -(q+1) S_1^{(3)} -(q+1) P^{(3)}
	\end{equation*}
	and
	\begin{equation*}
	\Div_3 (x_1^{q^3+q^2} w^{1-q^2}  v^{-q} ) = (q^3+q^2) Q^{(3)}   -(q^3+q^2) P^{(3)}.
	\end{equation*}
	Consider an  element $ f \in F^{(3)}  $, say
	\begin{align*}
	f  & \coloneqq
	x_1^{-s} \left(x_1^{q^2+q} w^{-q}  v^{-1} \right)^{\alpha} \left(x_1^{q^3+q^2} w^{1-q^2}  v^{-q} \right)^{\beta } \\
	&\;=x_1^{-s+(q^2+q)\alpha + (q^3+q^2)\beta }  \,\, w ^{-q\alpha + (1-q^2)\beta } \,\, v^{-\alpha -q\beta }  . 
	\end{align*}
	Then we obtain
	\begin{align*}
	& r Q^{(3)} +  s S_0^{(3)}+ t S_1^{(3)}+ u P^{(3)}+ \Div_3 (f)\\
	&  =(r-s) Q^{(3)} + (t- qs) S_1^{(3)}+ (u + q^2 s) P^{(3)}+(q^2+q )\alpha  Q^{(3)}-  (q+1)\alpha S_1^{(3)} \\
	&\phantom{=} \, -  (q+1)\alpha P^{(3)} +(q^3+q^2 )\beta  Q^{(3)} -  (q^3+q^2)\beta P^{(3)}\\
	&=\left(r-s+(q^2+q )\alpha+(q^3+q^2 )\beta \right)    Q^{(3)} + \widehat t S_1^{(3)} + \widehat u P^{(3)}.
	\end{align*}
	The desired conclusion then follows.
\end{proof}

\begin{lem}\label{lem:Omega_reduction}
	Suppose that $  t - q s  = \widehat t+  (q+1)\alpha$ with $ 0 \leqslant \widehat t < q+1 $ and  $ u + q^2 s -  (q+1)\alpha = (q^3+q^2) \beta + \widehat u$.  Let $ \widehat r = r-s+(q^2+q )\alpha+(q^3+q^2 )\beta $. Then we have $  \# \Omega_{r,s, t,u} =\# \Omega_{\widehat r,0 , \widehat t,\widehat u}  $ and
	\[ 	r+(q-1)s+(q-1)t+u =  \widehat{r}+ (q-1)   \widehat{t}+\widehat{u}.  \]
\end{lem}
\begin{proof}
	Proposition \ref{prop:div_equi} leads us to define the transformation $ {I} = i   + s- (q^2+q)\alpha - (q^3+q^2)\beta $, $ {J} =j + q\alpha - (1-q^2)\beta $ and ${K} =k+ \alpha + q\beta $, then $\Omega_{r,s,t,u }$ becomes
	\begin{align*}
	\Big\{ \,(I,J, K )\,\Big|\, 
	-\widehat r  &\leqslant I ,\,
	0  \leqslant I+ (q^2+q)K <   q^2+q , \\
	-\widehat t & \leqslant q I+ (q^2+q)J+ (q +1)K <-\widehat t+(q^2+q)   ,\\
	-   {\widehat u}  & \leqslant -q^2 I- (q^3-q)J- (q^3+q^2-q-1)K     \,\Big\},
	\end{align*}
	which implies the lemma.
\end{proof}

\begin{lem}\label{lem:set_inv}
	The lattice point set $ \Omega_{r,s,t,u} $ is symmetric, i.e., $ \# \Omega_{r,s,t,u}=\# \Omega_{u,t,s,r} $.
\end{lem}
\begin{proof}
	Let $I = -q^2 i - (q^3-q) j - (q^3+q^2-q-1)k $, $J = (q-1)i +(q^2-q-1)j+(q^2-1)k  $, $K = i + qj +q k$, then 
	$ \Omega_{r,s,t,u} $ is equivalent to 
	\begin{align*}
	\Big\{ \,(I,J,K )\,\Big|\, 
	-r &\leqslant -q^2 I -(q^3-q)J-(q^3+q^2-q-1)K,   \\
	-s & \leqslant qI+(q^2+q)J+(q+1)K <-s+  (q^2+q), \\
	-t & \leqslant I+(q^2+q)K <-t+(q^2+q)   ,    \,
	-u \leqslant  I  \,\Big\},
	\end{align*}
	which is identical with $ \Omega_{u,t,s,r} $.
\end{proof}

\begin{lem}\label{lem:Omega_large}
	Let $ r^{*}: =  q^3-q $, and $u^{*}:= q^3 + q^2   $.
	For  $ r\geqslant r^{*} $, $ u\geqslant u ^{*} $ and $0 \leqslant s,t<q+1$, we have
	$ \# \Omega_{ r, s , t, u } =  \# \Omega_{r ^{*},s, t, u ^{*} }  + (r-r^{*})+ (u-u^{*} ) $.
\end{lem}
\begin{proof}
	It is clearly that $ \Omega_{r^*,s,t,u}  \subseteq   \Omega_{r,s,t,u} $ for $ r \geqslant r^* $. We have by definition that  the complement of $ \Omega_{r^*,s,t,u} $ in $ \Omega_{r,s,t,u} $ is given by
	\begin{align}
	\Omega_{r,s,t,u}\setminus \Omega_{r^*,s,t,u} =\Big\{\, (i, j, k )\,\Big|\, 
	-r &\leqslant i \leqslant -(r^* +1)\label{eq:set_minus_1}, \tag{$A_1$} \\
	-s & \leqslant i+ (q^2+q)k < -s+ (q^2+q),\label{eq:set_minus_2}  \tag{$A_2$} \\
	-t & \leqslant q i+ (q^2+q)j+ (q +1)k <-t+(q^2+q)   ,\label{eq:set_minus_3}\tag{$A_3$}   \\
	-u & \leqslant -q^2 i- (q^3-q)j- (q^3+q^2-q-1)k \label{eq:set_minus_4} \tag{$A_4$}   \,\Big\}. 
	\end{align}
	Firstly, we shall show that Condition (\ref{eq:set_minus_4}) is invalid under the other conditions. Let $ N: = q i+ (q^2+q)j+ (q +1)k $. We obtain by Condition (\ref{eq:set_minus_3}) that 
	\[ N < -t+(q^2+q) \leqslant q^2+q ,\]
	and by Condition (\ref{eq:set_minus_2}) that 
	\[   i < -s+ (q^2+q)-(q^2+q)k \leqslant (q^2+q)(1-k). \] 	
	The right hand side of (\ref{eq:set_minus_4}) verifies that
	\begin{align}
	&-q^2 i  - (q^3-q)j- (q^3+q^2-q-1)k  \nonumber
	\\
	& =-q  i   - (q-1)N  - (q^3 -q)k \nonumber \\
	& >  (q^2+q)(k - q+ 1 ) - u^* .\label{eq:invalid}
	\end{align}
	It follows from  Conditions (\ref{eq:set_minus_1}) and (\ref{eq:set_minus_2}) that
	\[ k = \ceil{\frac{-s -i}{q^2+q }} \geqslant \ceil{\frac{r^*-q+1 }{q^2+q}} = q-1 .  \] 
	So we can omit Condition (\ref{eq:set_minus_4}) according to Equation (\ref{eq:invalid}). Since Conditions (\ref{eq:set_minus_2}) and (\ref{eq:set_minus_3}) only tell  us that $j$ and $k$ depend on $i$; namely
	\[ k=  \ceil{\frac{-s-i}{q^2+q}}, \quad j=\ceil{\frac{-t-qi-(q+1)k}{q^2+q}} .
	\]
	So we can omit them too.
	Therefore, the set $\Omega_{r,s,t,u}\setminus \Omega_{r^*,s,t,u}$ is equivalent to the lattice point set 
	\begin{equation*}
	\Big\{ \, i \,\Big|\,  -r  \leqslant  i \leqslant -(r^*+1)\,  \Big\}.  
	\end{equation*}
	It follows that
	\begin{equation}\label{eq:Omega_r}
	\# \Omega_{r,s,t,u}-\#  \Omega_{r^*,s,t,u} = \# \left(\Omega_{r,s,t,u}\setminus \Omega_{r^*,s,t,u} \right) = r-r^* .  
	\end{equation}
	Similarly, by Lemma \ref{lem:set_inv} we get 
	\begin{equation}\label{eq:Omega_u}
	\# \Omega_{r ,s,t,u}-\#  \Omega_{r ,s,t,u^* } = u-u^* . 
	\end{equation}
	Combining Equations (\ref{eq:Omega_r}) and (\ref{eq:Omega_u}), we obtain 
	\begin{align*}
	\# \Omega_{r ,s,t,u}-\#  \Omega_{r^*,s,t,u^* } & = \#\Omega_{r,s,t,u}-   \#\Omega_{r^*,s,t,u  }+\#\Omega_{r^*,s,t,u}-\#  \Omega_{r^*,s,t,u^* } \\
	& = (r-r^*) + (u-u^*),
	\end{align*}
	which implies the lemma.
\end{proof}

\begin{lem} \label{lem:Psi_a}
	Let $a$ and $u$ be two non-negative  integers  and let  $ {\Psi}_{a}(u)  $ be a lattice point set 
	\begin{align}
	{\Psi}_{a}(u): =\Big\{\, ( {i}, {j}) \,\Big|\,
	& \, 0 \leqslant   {i} < q^2+q, \label{eq:Psi_a_1}\tag{$B_1$} \\
	& \, 0 \leqslant q {i} +(q^2+q) {j}   < q^2+q , \label{eq:Psi_a_2} \tag{$B_2$}\\
	& \, q^2   {i} + (q^3-q)    {j}     \leqslant   {u}   -(q^3+q^2)a    \, \Big\}.\nonumber
	\end{align}
	Then for $u \geqslant u^*:=q^3+q^2 $, we have
	\[   \sum_{a=0}^{\infty}\# \Psi_a (u)=  ( -q^2 +  q +2  )/2 +  \floor{\frac{u}{q}}   . \]
\end{lem}
\begin{proof}
	According to Condition (\ref{eq:Psi_a_1}), we may assume $ i=c(q+1)+d $, with $0 \leqslant c < q$, and $0 \leqslant d < q+1$. Then Condition (\ref{eq:Psi_a_2}) can be restated as
	\[ 0 \leqslant  q d + (q^2+q)(j+c) < q^2+q ,\]
	which is equivalent to  $j=-c$. Hence,  we get
	\begin{align*}
	{\Psi}_{a}(u) \cong \Big\{ \,( c,d )\,\Big|\, 
	0 \leqslant c < q,  \,  0 \leqslant d   < q+1, \,
   {(q^2+q)(qa + c )}  +	q^2 d    \leqslant   {u}  \,\Big\}.
	\end{align*}
	
	Let $e= qa+ c$. We find that	$ {\Psi}_{a}(u)  $ is equivalent to 
	\begin{align}
	{\Psi}_{a}^{\prime} (u) := \Big\{ \,(d,e) \,\Big|\,  
	\,q a & \leqslant e < q+qa, \label{eq:a1} \tag{$ C_1 $} \\
	0 & \leqslant d   < q+1,\label{eq:a2}  \tag{$ C_2 $}  \\
	& \hspace{-10pt} (q^2   +q)e   +	q^2 d    \leqslant   {u}    \,\Big\}.\label{eq:a3}  \tag{$ C_3 $} 
	\end{align}
	For abbreviation we set $\Psi (u) := \bigcup_{a=0}^{\infty }{\Psi}_{a}^{\prime}(u)    $.
	Note that Conditions (\ref{eq:a2}) and (\ref{eq:a3}) are  independent of $ a$. So we can combine Condition (\ref{eq:a1}) for $ a \geqslant 0$. We can write down Condition (\ref{eq:a1})  for various $ a $  as follows
	\begin{align*}
	0 \leqslant e &< q\hphantom{1} \quad  \text{for } a=0,\\
	q \leqslant e &<2 q \quad \text{for } a=1,\\
	2q \leqslant e &<3 q \quad \text{for } a=2,\\
	& \cdots 
	\end{align*}
	This gives a total condition   $ e \geqslant 0$ and   a new expression for $ \Psi (u) $
	\begin{align*}
	\Psi(u) := \Big\{ \,(d,e)\,\Big|\,    0 \leqslant e ,   \, 0 \leqslant d   < q+1, \, {(q^2+q)e}  +	q^2 d    \leqslant   {u}    \,\Big\}.
	\end{align*}
	
	Using the fact that $ \Psi(u^* ) $ contains exactly the points of the triangle $ \triangle OAB $ except the vertex $B$, one can easily show that
	\begin{equation}
	\# \Psi(u^* )   = \# \triangle OAB -1   =  ( q^2 +3 q +2  )/2.
	\end{equation}
	Denote by $ M_{\lambda}  $ the lattice point set 
	\begin{align*}
	M_{\lambda} := \Big\{ \,(d,e)\,\Big|\,    0 \leqslant d   < q+1,  \,  {(q^2+q)e}  +	q^2 d  = {\lambda}  \, \Big\}.
	\end{align*}
	Then we obtain
	\[
	\Psi(u)=\Psi(u^* ) \bigcup  \left(  \bigcup_{{\lambda}=u^*+1  }^{u} M_{\lambda} \right) .
	\]
	The set $ M_{\lambda} $ is non-empty if and only if $   q \mid \lambda $. For such $\lambda $,  there exists only one lattice point on   $ M_{\lambda } $.
	So we have
	\[  
	\# M_{\lambda}  =\begin{cases}
	1  & \text{for $ q \mid \lambda $, }\\
	0  & \text{for $ q \nmid  \lambda  $. } 
	\end{cases}
	\]
	Then we get
	\[ \# \Psi(u) =    \# \Psi(u^* )+  \floor{\frac{u-u^* }{q}}  =  ( -q^2 +  q +2  )/2 +  \floor{\frac{u}{q}} ,  \]
	which finishes the proof.
\end{proof}
\begin{figure}[H]
	\centering
	\begin{tikzpicture}[scale=0.40]
	\path [fill=orange] (0,0)--(0,4)--(5,0)--(0,0);
	\draw [->](0,-1)--(0,8);
	\draw [dashed] (5,-1)--(5,8);
	\draw [->](-1,0)--(7,0);
	\draw [blue](0,4)--(5,0);
	\draw [red](0,7)--(5,3);
	\draw [fill] (0,0) circle [radius=0.1];
	
	\draw[color=black] node [yshift=-1ex,xshift=+1ex]  at (5,0) { $ B $ };		
	\draw[color=black] node [yshift=0ex,xshift=-1ex] at (0,4){ $ A $ };
	\draw[color=black] node [yshift=-1ex,xshift=-1ex] at (0,0){ $ O $ };
	\draw[color=black] node   at (-1.8,2){ $ \Psi(u^*) $ };
	\draw [->]   (-0.8, 1.6)--(0.5,1.2);
	\draw[color=black] node   at (3,6.8){ $ M_{\lambda} $ };
	\draw [->]   (2.5, 6.5)--(1.7,5.7);
	\draw[color=black] node [yshift=0ex,xshift=+1ex] at (7,0){ $ d $ };
	\draw[color=black] node [yshift=0ex,xshift=-1ex] at (0,8){ $ e $ };
	\draw [fill] (0,0) circle [radius=0.1];
	\draw [fill] (0,1) circle [radius=0.1];
	\draw [fill] (0,2) circle [radius=0.1];
	\draw [fill] (0,3) circle [radius=0.1];
	\draw [fill] (0,4) circle [radius=0.1];
	\draw [fill] (0,5) circle [radius=0.1];
	\draw [fill] (0,6) circle [radius=0.1];
	\draw [fill] (0,7) circle [radius=0.1];
	\draw [fill] (1,0) circle [radius=0.1];
	\draw [fill] (1,1) circle [radius=0.1];
	\draw [fill] (1,2) circle [radius=0.1];
	\draw [fill] (1,3) circle [radius=0.1];
	\draw [fill] (1,4) circle [radius=0.1];
	\draw [fill] (1,5) circle [radius=0.1];
	\draw [fill] (1,6) circle [radius=0.1];
	
	\draw [fill] (2,0) circle [radius=0.1];
	\draw [fill] (2,1) circle [radius=0.1];
	\draw [fill] (2,2) circle [radius=0.1];
	\draw [fill] (2,3) circle [radius=0.1];
	\draw [fill] (2,4) circle [radius=0.1];
	\draw [fill] (2,5) circle [radius=0.1];
	
	\draw [fill] (3,0) circle [radius=0.1];
	\draw [fill] (3,1) circle [radius=0.1];
	\draw [fill] (3,2) circle [radius=0.1];
	\draw [fill] (3,3) circle [radius=0.1];
	\draw [fill] (3,4) circle [radius=0.1];
	
	\draw [fill] (4,0) circle [radius=0.1];
	\draw [fill] (4,1) circle [radius=0.1];
	\draw [fill] (4,2) circle [radius=0.1];
	\draw [fill] (4,3) circle [radius=0.1];
	
	\end{tikzpicture}
	\protect\caption{the lattice point set $\Psi(u)$ with $q=4$}
\end{figure}
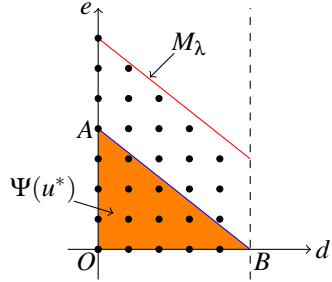

\begin{lem} \label{lem:Phi_b}
	Let $b $ be an integer such that $ 0 \leqslant b < q-1 $ and let $t$ be an non-negative integer.  Suppose that $ 	{\Phi}_{ b }^{(1)} $ is a lattice point set
	\begin{align*}
	{\Phi}_{ b }^{(1)}:=\Big\{\, ( {i}, {j}) \,\Big|\,
	\,    0  \leqslant    {i} < q^2+q, \,\, - t- (q+1)b \leqslant  q  {i} +(q^2+q) {j}  < 0 \, \Big\}.
	\end{align*}
	Then  we obtain
	\[    \sum_{b=0}^{q-1} \# {\Phi}_{ b }^{(1)}=q^2 (q-1)/2 +q t  . \]
\end{lem}
\begin{proof}
	Denote by $L_\lambda $ the lattice point set 
	\begin{align*}
	L_{\lambda}:=\Big\{\,(i,j)\,\Big|\,0 \leqslant i< q^2+q, \,\,
	qi+(q^2+q)j = -  \lambda  \, \Big\}.
	\end{align*} 
	The set $ L_{\lambda} $ is non-empty if and only if $ q \mid \lambda   $. For such $ \lambda $,  there exist $q$ lattice points on   $ L_{\lambda } $.
	So we have
	\[  
	\# L_{\lambda}  =\begin{cases}
	q  & \text{for $ q \mid \lambda $, }\\
	0  & \text{for $ q \nmid \lambda $. } 
	\end{cases}
	\]
	Using the fact that
	\[  {\Phi}_{ b }^{(1)}= \bigcup_{\lambda=1}^{t  + (q+1)b}{L_\lambda} ,\]
	we obtain
	\begin{align*}
	\sum_{b=0}^{q-1} \# {\Phi}_{ b }^{(1)} & =\sum_{b=0}^{q-1} \left(\sum_{\lambda =1}^{ t +(q+1)b} \# L_{\lambda} \right) \\
	&=\sum_{b=0}^{q-1} \floor{\frac{{t +(q+1)b}}{q}} q \\
	&=q \sum_{b=0}^{q-1} b +q \sum_{b=0}^{q-1} \floor{\frac{{t + b}}{q}}   \\
	&=q^2 (q-1)/2 +q t  . \\
	\end{align*}
	This completes the proof.
\end{proof}

\begin{lem} \label{lem:Omega_main}
	Let $r^{*}= q^3-q$, $u^{*}:= q^3 + q^2 $. Then we have for $0\leqslant t < q+1$  
	\[ \# \Omega_{ r^*, 0 , t, u^{*} } = 1-g^{(3)} +r^* + (q-1)t +u^{*}  .\] 
\end{lem}
\begin{proof}
	It is important to write $k=-q a + b$ with $a, b \in \mathbb{Z}$ and $0 \leqslant b <q  $.  We define
	\[\Theta_{k}:=\Theta_{a,b}:=\Big\{\,(i,j)\,\Big|\,(i,j,-qa+b)\in \Omega_{r^{*}, 0, t , u}\,\Big\} . \]
	Precisely speaking, $ \Theta_{k} $ is given by 
	\begin{align}
	\Theta_{k} =\Big\{ \,(i, j  )\,\Big|\, 
	-r^{*} &\leqslant i,\label{Omega_1} \tag{$ D_1 $}\\
	0 & \leqslant i+ (q^2+q)k < (q^2+q),\label{Omega_2}\tag{$ D_2 $}\\
	-t  & \leqslant q i+ (q^2+q)j+ (q +1)k <-t +(q^2+q)\tag{$ D_3 $}   ,\\
	-u & \leqslant -q^2 i- (q^3-q)j- (q^3+q^2-q-1)k     \, \Big\}.\label{Omega_4}\tag{$ D_4 $}
	\end{align}
	So we have $k = \ceil{\frac{-i}{q^2+q}} \leqslant \ceil{\frac{r^*}{q^2+q}}  = q-1$ and therefore  $\Theta_k$ is empty for $k > q-1$. While for $k \leqslant q-1$, Condition (\ref{Omega_2}) implies that
	\[  i\geqslant -(q^2+q)(q-1)= -r^*. \] 
	So Condition (\ref{Omega_1}) is invalid. Then we have 
	\begin{equation}\label{eq:Omega_sum}
	\Omega_{r^*, 0, t, u}= \bigcup_{k=-\infty}^{q-1} \Theta_{k}  = \bigcup_{b=0}^{q-1}\bigcup_{a=0}^{\infty} \Theta_{a,b}  .
	\end{equation}

	Let $ \widetilde{i} : = i + (q^2+q) k   $, $ \widetilde{j} := j - q k - a $, and $u =u^{*} + l (q^2+q) $ with $ 0 \leqslant l < q $. Then $\Theta_{a,b}$ becomes
	\begin{align}
	{\Phi}_{a,b,l}=\Big\{\, (\widetilde{i},\widetilde{j}) \,\Big|\,
	& 0 \leqslant   \widetilde{i}  < q^2+q, \nonumber \\
	& - t - (q+1)b   \leqslant q \widetilde{i} +(q^2+q)\widetilde{j}   <  (q^2+q)  - \left(t  + (q+1)b\right), \label{eq:Phi_2} \tag{$ E_2 $}\\
	& q^2  \widetilde{i} + (q^3-q)    \widetilde{j}     \leqslant   u^{*}+ (q^2+q)l + (q+1)b -(q^3+q^2)a  \, \Big\}. \nonumber
	\end{align}
	By Equation (\ref{eq:Omega_sum}) and Lemma \ref{lem:Omega_large} we obtain
	\begin{equation} \label{eq:Omega_sum_l}
	\sum_{b=0}^{q-1}\sum_{a=0}^{\infty} \# \Phi_{a,b,l}=\# \Omega_{r^*, 0,t ,u^{*}+ (q^2+q)l} = \# \Omega_{r^*, 0,t ,u^{*}}+ (q^2+q)l.
	\end{equation}
	Condition (\ref{eq:Phi_2}) can be split into two parts; namely,
	\begin{align*}
	\text{Part (1): } &-\left(t  +(q+1)b \right) \leqslant q\widetilde{i} + (q^2+q)\widetilde{j} < 0 , \\
	\intertext{and} 
	\text{Part (2): }& 0 \leqslant q\widetilde{i} + (q^2+q)\widetilde{j} <  (q^2+q)  - t - (q+1)b  .
	\end{align*}   
	Then we obtain
	\[ \Phi_{a,b,l}= \Phi^{(1)} _{a,b,l} \bigcup \Phi^{(2)} _{a,b,l} ,\]
	where 
	\begin{align}
	{\Phi}_{a,b,l}^{(1)}=\Big\{\, (\widetilde{i},\widetilde{j}) \,\Big|\,
	& 0 \leqslant   \widetilde{i} < q^2+q, \nonumber \\
	& - t - (q+1)b  \leqslant q \widetilde{i} +(q^2+q)\widetilde{j}   < 0, \nonumber  \\
	&  q^2  \widetilde{i} + (q^3-q)    \widetilde{j}     \leqslant   u^{*} +(q^2+q)l + (q+1)b  -(q^3+q^2)a    \, \Big\}, \label{eq:Phi_3} \tag{$ F_3 $}
	\end{align}
	and
	\begin{align*}
	{\Phi}_{a,b,l}^{(2)}=\Big\{\, (\widetilde{i},\widetilde{j})\,\Big|\,
	& 0 \leqslant   \widetilde{i}  < q^2+q, \\
	& 0  \leqslant q \widetilde{i} +(q^2+q)\widetilde{j}   <  (q^2+q)  - \left(t + (q+1)b\right),  \\
	& q^2  \widetilde{i} + (q^3-q)    \widetilde{j}     \leqslant   u^{*} +(q^2+q)l + (q+1)b  -(q^3+q^2)a    \, \Big\}.
	\end{align*}
	Let $  \widehat{j}= \widetilde{j} +1 $, then $ {\Phi}_{a,b,l}^{(1)} $ is equivalent to 
	\begin{align*}
	{\widehat \Phi}_{a,b,l}^{(1)}=\Big\{ \,(\widetilde{i},\widehat{j}) \,\Big|
	\,  & 0 \leqslant   \widetilde{i}  < q^2+q, \\
	& (q^2+q)- \left(t + (q+1)b\right)  \leqslant q \widetilde{i} +(q^2+q)\widehat{j}   < q^2+q  ,  \\
	& q^2  \widetilde{i} + (q^3-q)    \widehat{j}     \leqslant   u^{*} +(q^2+q)(l+q-1) + (q+1)b -(q^3+q^2)a   \, \Big\}. 
	\end{align*}
	We set 
	\begin{align*}
	{\Psi}_{a,b,l} = \Big\{ \,(\widetilde{i},\widetilde{j}) \,\Big|\,  &
	0 \leqslant   \widetilde{i}  < q^2+q, \\
	& 0  \leqslant q \widetilde{i} +(q^2+q)\widetilde{j}   <  (q^2+q)   ,\\
	& q^2  \widetilde{i} + (q^3-q)    \widetilde{j}     \leqslant   u^{*}+(q^2+q)l + (q+1)b  -(q^3+q^2)a   \, \Big\}.
	\end{align*}
	Then we can split the set $ {\Psi}_{a,b,l} $   as follows
	\begin{align*}
	\Psi_{a,b,q-1} & =   \widehat\Phi_{a,b,0}^{(1)}    \cup   \Phi_{a,b,q-1}^{(2)},   \\
	\Psi_{a,b,0} &=  \widehat\Phi_{a+1,b,1}^{(1)} \cup   \Phi_{a,b,0}^{(2)}, \\
	\Psi_{a,b, 1} &= \widehat\Phi_{a+1,b,2}^{(1)} \cup   \Phi_{a,b,1}^{(2)},    \\
	& \cdots, \\
	\Psi_{a,b,q-2} &=\widehat \Phi_{a+1,b,q-1}^{(1)} \cup   \Phi_{a,b,q-2}^{(2)}  . \\
	\end{align*}
	This implies 
	\[  
	\sum_{l=0}^{q-1} \# \Psi_{a,b,l} =  \# \widehat \Phi^{(1)}_{a ,b,0}+ \sum_{l=1}^{q-1} \# \widehat \Phi^{(1)}_{a+1,b,l}+\sum_{l=0}^{q-1} \# \Phi^{(2)}_{a,b,l}  .  
	\]
	Using this equation, we conclude that
	\begin{align}
	\sum_{b=0}^{q-1}\sum_{l=0}^{q-1}\sum_{a=0}^{\infty} \# \Phi_{a,b,l} & = \sum_{b=0}^{q-1}\sum_{l=0}^{q-1}\sum_{a=0}^{\infty} \# \Phi_{a,b,l} ^{(1)} + \sum_{b=0}^{q-1}\sum_{l=0}^{q-1}\sum_{a=0}^{\infty} \# \Phi_{a,b,l} ^{(2)} \nonumber \\
	& = \sum_{b=0}^{q-1} \sum_{a=0}^{\infty} \# \widehat \Phi_{a,b,0} ^{(1)} +   
	\sum_{b=0}^{q-1}\sum_{l=1}^{q-1}\left(    \# \widehat \Phi_{0,b,l} ^{(1)} + \sum_{a=1}^{\infty}\# \widehat \Phi_{a,b,l} ^{(1)} 
	\right)      + \sum_{b=0}^{q-1}\sum_{l=0}^{q-1}\sum_{a=0}^{\infty} \# \Phi_{a,b,l} ^{(2)}\nonumber \\
	& =  \sum_{b=0}^{q-1}\sum_{l=0}^{q-1}\sum_{a=0}^{\infty} \# \Psi_{a,b,l} +  \sum_{b=0}^{q-1}\sum_{l=1}^{q-1}  \# \Phi_{0,b,l} ^{(1)} . \label{eq:sum_all}
	\end{align}
	
	In the following we give the explicit calculation of Equation \eqref{eq:sum_all}. 
	From the proof of Lemma \ref{lem:Omega_large}, we see that when $ a=0 $  the condition (\ref{eq:Phi_3}) is invalid.
	Applying Lemma \ref{lem:Phi_b} we get
	\[  \sum _{b=0}^{q-1} \# \Phi_{0,b,l} ^{(1)} = {{q t +q^2 (q-1)/2 }} ,  \]
	and then we have
	\begin{equation} \label{eq:sum_right_2}
	\sum_{b=0}^{q-1}\sum_{l=1}^{q-1}  \# \Phi_{0,b,l} ^{(1)}  = (q-1)qt  +q^2 (q-1)^2/2. 
	\end{equation}
	By Lemma \ref{lem:Psi_a} we have
	\[   \sum_{a=0}^{\infty} \# \Psi_{a,b,l}   =   ( -q^2 +  q +2  )/2 +  \floor{\frac{u^{* } + (q^2+q)l +(q+1)b }{q}},  \]
	and therefore
	\begin{align}
	\sum_{b=0}^{q-1}\sum_{l=0}^{q-1}\sum_{a=0}^{\infty} \# \Psi_{a,b,l}  &= \sum_{b=0}^{q-1}\sum_{l=0}^{q-1}\left( ( -q^2 +  q +2  )/2 +  \floor{\frac{u^{* } + (q^2+q)l +(q+1)b }{q}}  \right) \nonumber\\
	&=q^2 \left(   -q^2 +  q +2    \right)/2 + \sum_{b=0}^{q-1}\sum_{l=0}^{q-1} ((q+1)l+b)  + \sum_{l=0}^{q-1} \sum_{b=0}^{q-1} \floor{\frac{u^*+b}{q}}  \nonumber  \\
	&=q^2 \left(  -q^2 +  q +2  \right) /2 + q\sum_{b=0}^{q-1} b + (q+1)q\sum_{l=0}^{q-1}l  + \sum_{l=0}^{q-1} u^* \nonumber \\
	&=q^2 \left(  -q^2 +  q +2   \right)/2  + (q^3-q^2) /2  + (q^4-q^2) /2  +q u^* \nonumber\\
	&=q^3 +q u^{* } . \label{eq:sum_right_1}
	\end{align}
	On the other hand, it follows from Equation \eqref{eq:Omega_sum_l} that
	\begin{align}
	\sum_{l=0}^{q-1} \sum_{b=0}^{q-1}\sum_{a=0}^{\infty} \# \Phi_{a,b,l} & =\sum_{l=0}^{q-1} \left(  \# \Omega_{r^* ,0,t ,u^{*}} +  (q^2+q)l \right) \nonumber  \\
	& = q  \# \Omega_{r^* ,0,t ,u^{*}} + (q^4-q^2)/2.  \label{eq:sum_left}
	\end{align}
	By Proposition \ref{prop:div_deg}, the genus is
	\[
	g^{(3)}= q^3-2q+1 =r^*-q+1.
	\]
	Using Equations \eqref{eq:sum_all}, \eqref{eq:sum_right_2}, \eqref{eq:sum_right_1} and \eqref{eq:sum_left}, we conclude that
	\begin{align*}
	\# \Omega_{r^{*},0,t , u^{*}}&  = q + (q-1)t  +u ^{*}\\
	& =  1-g^{(3)} + r^{*} + (q-1)t  +u ^{*}.
	\end{align*} 
	This finishes the proof.
\end{proof}
Now Proposition \ref{prop:Omega_num} follows from Lemmas  \ref{lem:Omega_reduction},  \ref{lem:Omega_large}  and  \ref{lem:Omega_main}.
\begin{proof}[Proof of Proposition \ref{prop:Omega_num}]
	Let $ r \geqslant R $. The cardinality of $ \Omega_{r ,s,t , u } $ is given as follows: 
	\begin{align*} 
	\# \Omega_{r ,s,t , u }&  = \# \Omega_{\widehat r,0 , \widehat t,\widehat u} \quad \quad \textup{(by Lemma \ref{lem:Omega_reduction})}\\
	& =  \# \Omega_{r^* ,0, \widehat t , u^* } + (\widehat r -r^* ) +(\widehat u - u^*) \quad \quad \textup{(by Lemma \ref{lem:Omega_large})}  \\
	& = 1- g^{(3)}+ r^* + (q-1) \widehat t + u^* + (\widehat r -r^* ) +(\widehat u - u^*) \quad \quad \textup{(by Lemma \ref{lem:Omega_main})} \\
	& = 1- g^{(3)}+ \widehat r + (q-1) \widehat t +  \widehat u    \\ 
	& = 1- g^{(3)}+ r  + (q-1)s + (q-1)   t + u.  \quad \quad \textup{(by Lemma \ref{lem:Omega_reduction})}  
	\end{align*}
	This finishes the proof of Proposition \ref{prop:Omega_num}.
\end{proof}

\section{Divisors on a line}\label{sec:divi}

Let $  \alpha_\mu $ be distinct elements in  $\mathbb{F}_{q^2}^{*} $ such that $  \alpha_\mu^{q-1}+1 =0 $ for $ 1 \leqslant \mu \leqslant q-1 $. 
Recall that the divisors $ S_0^{(3)} $ and $ S_1^{(3)} $ are decomposed as follows:
\[ S_0^{(3)} = \sum_{\mu = 1 }^{q-1} S_{0,\mu}^{(3)}, \quad S_1^{(3)} = \sum_{\mu = 1 }^{q-1} S_{1,\mu}^{(3)},
\]
where $ S_{0,\mu}^{(3)} $ and $ S_{1,\mu}^{(3)} $ are rational places of
$ z_3-\alpha_\mu $ and $ z_2-\alpha_\mu $, respectively. 
By Definition \ref{def:tow} that $ z_2=x_1 x_2 $ and $ z_3=x_2 x_3 $, we have 
\begin{align*} 
\Div_3(z_2) &=  (q+1)Q^{(3)} + (q+1)S_0^{(3)} -(q^2+q)P^{(3)},\\
\Div_3(z_3) &=  (q^2+q)Q^{(3)} - (q+1)S_1^{(3)} -(q+1)P^{(3)}, \\ 
\Div_2(z_2 - \alpha_ \mu ) &=  (q+1) S_{1,\mu}^{(2)} -(q+1)P^{(2)} ,  \\	
\Div_3(z_2- \alpha_\mu ) &=  (q^2+q) S_{1,\mu}^{(3)} -(q^2+q)P^{(3)},\\ \Div_3(z_3 - \alpha_\mu) &= (q^2+q ) S_{0,\mu}^{(3)} - (q+1)S_1^{(3)}-(q+1)P^{(3)}.
\end{align*}
In this section we denote $ G: =  r Q^{(3)} + \sum_{\mu = 1}^{q-1} s_\mu S_{0, \mu }^{(3)} + \sum_{\nu = 1}^{q-1} t_\nu S_{1, \nu }^{(3)}  + u P^{(3)}$ and define a lattice point set $ \Omega_{r,s_1, \cdots, s_{q-1} ,t_1,\cdots, t_{q-1}, u} $ as follows:	
\begin{align}\label{eq:Omega2_0}
\Big\{ \, &  (i, j_1,\cdots,j_{q-1} , k_1,\cdots,k_{q-1} )   \, \Big|\,    -r \leqslant i ,  \nonumber \\
&-s_\mu \leqslant  i + (q^2+q) k_\mu < -s_\mu+ (q^2+q) \textup{ for } \mu =1, \cdots, q-1,   \nonumber \\
&-t_\nu \leqslant  qi + (q^2+q) j_\nu - (q+1) \sum_{\mu = 1}^{q-1} k_\mu < -t_\nu + (q^2+q) \textup{ for } \nu =1, \cdots, q-1,  \nonumber \\
&-u \leqslant  -q^2 i - (q^2+q) \sum_{\nu = 1}^{q-1} j_\nu -(q+1) \sum_{\mu = 1}^{q-1}  k_\mu    \,\Big\},  
\end{align}
or equivalently it is written as
\begin{align}\label{eq:Omega2}
  \Big\{ \, &  (i, j_1,\cdots,j_{q-1} , k_1,\cdots,k_{q-1} )   \, \Big|\,   -r \leqslant i  , \nonumber \\
&  k_\mu = \ceil{ \dfrac{-i-s_{\mu}}{q^2+q}} \textup{ for } \mu =1, \cdots, q-1,  \nonumber \\
&j_\nu= \ceil{\dfrac{-qi-t_{\nu}}{q^2+q} +\dfrac{1}{q} \sum_{\mu=1 }^{q-1} k_\mu } \textup{ for } \nu =1, \cdots, q-1,  \nonumber \\
&-u \leqslant  -q^2 i - (q^2+q) \sum_{\nu = 1}^{q-1} j_\nu -(q+1) \sum_{\mu = 1}^{q-1}  k_\mu    \,\Big\}.  
\end{align}

The main results of this section is given below.
\begin{prop} \label{prop:Omega2}
	There exists a constant $ R           $ depending on $s_\mu $, $ t_\nu $ and $u $, such that for $ r \geqslant R $,
	\[ \#  \Omega_{r,s_1, \cdots, s_{q-1} ,t_1,\cdots, t_{q-1}, u} = 1-g^{(3)}+ r +  \sum_{\mu =1 }^{q-1} s_{\mu } +   \sum_{\nu =1 }^{q-1} t_ {\nu } +u .\]
\end{prop}
We will prove Proposition \ref{prop:Omega2} after some preparations. Before we do so, we can determine a basis for the Riemann-Roch space $ \mathcal{L}(G) $ applying Proposition \ref{prop:Omega2}.
\begin{prop}\label{prop:basis2}
	The elements 
	$ 	x_1^{i} \prod_{\nu=1}^{q-1} {(z_2 -\alpha_\nu)^ {j_\nu} } \prod_{\mu=1}^{q-1} {(z_3 -\alpha_\mu)^ {k_\mu} } $ with $ (i, j_1,\cdots,j_{q-1} , k_1,\cdots,k_{q-1} ) \in \Omega_{r,s_1, \cdots, s_{q-1} ,t_1,\cdots, t_{q-1}, u } $ form a basis of $ \mathcal{L}(G) $.
\end{prop}
\begin{proof}
	It is easy to compute the divisors of these elements in $ F^{(3)} $:
	\begin{align*}
	& \Div_3 \left(	 x_1^{i} \prod_{\nu=1}^{q-1} {(z_2 -\alpha_\nu)^ {j_\nu} } \prod_{\mu=1}^{q-1} {(z_3 -\alpha_\mu)^ {k_\mu} }  \right)\\
	&  =  i Q^{(3)} + \sum_{\mu=1}^{q-1}\left( i+(q^2+q) k_\mu  \right) S_{0,\mu}^{(3)} + \sum_{\nu=1}^{q-1} \left( qi + (q^2+q)  j_\nu -(q+1)\sum_{\mu=1}^{q-1} k_\mu  \right) S_{1,\nu}^{(3)} \\
	& \quad + \left( -q^2 i -(q^2+q) \sum_{\nu=1}^{q-1} j_\nu -(q+1) \sum_{\mu=1}^{q-1} k_\mu \right) P^{(3)} .
	\end{align*}
	Similar  to the proof of Proposition \ref{prop:basis},  we achieve the proof by using Proposition \ref{prop:Omega2}.
\end{proof}

Now we give some lemmas describing the lattice point set $ \Omega_{r ,s_1,\cdots,s_{q-1}, t_1, \cdots, t_{q-1}, u } $.
\begin{lem}  \label{lem:set_inv2}
	Denote $ s_0:=r  $ and $ t_0:=u $. The lattice point set $ \Omega_{s_0,s_1, \cdots, s_{q-1} ,t_1,\cdots, t_{q-1},t_0} $ satisfies the following identities:	
	\[\# \Omega_{s_0,s_1,\cdots, s_{q-1}, t_1,\cdots, t_{q-1}, t_0} = \# \Omega_{s_0', s_1',\cdots, s_{q-1}', t_1',\cdots, t_{q-1}', t_0'} , \]
	and \[\# \Omega_{s_0,s_1,\cdots, s_{q-1}, t_1,\cdots, t_{q-1}, t_0} = \# \Omega_{  t_0',t_1',\cdots, t_{q-1}', s_1',\cdots, s_{q-1}',s_0'} , \]
	where the sequences $ (s_i')_{i=0}^{q-1}  $ and $ (t_i')_{i=0}^{q-1}  $ are permutations of $ (s_i)_{i=0}^{q-1}  $ and $ (t_i)_{i=0}^{q-1}  $, respectively.
\end{lem}
\begin{proof}
	Recall that the lattice point set $ \Omega_{s_0,s_1, \cdots, s_{q-1} ,t_1,\cdots, t_{q-1}, t_0} $  is defined by
	\begin{align}
	\Big\{ \, &  (i, j_1,\cdots,j_{q-1} , k_1,\cdots,k_{q-1})   \, \Big|\,    -s_0 \leqslant i ,  \nonumber \\
	&  k_\mu = \ceil{ \dfrac{-i-s_{\mu}}{q^2+q}} \textup{ for } \mu =1, \cdots, q-1,  \nonumber \\
	&-t_\nu \leqslant  qi + (q^2+q) j_\nu - (q+1) \sum_{\mu = 1}^{q-1} k_\mu < -t_\nu + (q^2+q) \textup{ for } \nu =1, \cdots, q-1,  \nonumber \\
	&-t_0 \leqslant  -q^2 i - (q^2+q) \sum_{\nu = 1}^{q-1} j_\nu -(q+1) \sum_{\mu = 1}^{q-1}  k_\mu   \label{eq:Omi_t}\tag{$H_4$}   \,\Big\}.  
	\end{align}
	Write $ -t_0= qi+(q^2+q)j_0-(q+1) \sum_{\mu=1}^{q-1} k_{\mu}-\eta $ for integers $ j_0 $ and $ \eta $, where $ 0 \leqslant \eta < q^2+q $. Then the last condition \eqref{eq:Omi_t} gives that
	\[i+\sum_{\nu = 0}^{q-1} j_\nu \leqslant \dfrac{\eta}{q^2+q}<1,\]
	and consequently
	\[i+\sum_{\nu = 0}^{q-1} j_\nu \leqslant 0 .\]
	Namely the set
	$ 	\Omega_{s_0,s_1, \cdots, s_{q-1} ,t_1,\cdots, t_{q-1}, t_0}  $ becomes
	\begin{align}\label{eq:Ome_t2}
	\Big\{ \, &  (i, j_0, j_1,\cdots,j_{q-1} , k_1,\cdots,k_{q-1} )   \, \Big|\,    -s_0 \leqslant i ,  \nonumber \\
	&  k_\mu = \ceil{ \dfrac{-i-s_{\mu}}{q^2+q}} \textup{ for } \mu =1, \cdots, q-1,   \nonumber  \\
	&-t_\nu \leqslant  qi + (q^2+q) j_\nu - (q+1) \sum_{\mu = 1}^{q-1} k_\mu < -t_\nu + (q^2+q) \textup{ for } \nu =0, \cdots, q-1,   \nonumber  \\
	&i+\sum_{\nu = 0}^{q-1} j_\nu\leqslant 0   \,\Big\}.  
	\end{align}
	So the number of the lattice points does not depend on the order of $t_{\nu}$ with $ 0 \leqslant \nu \leqslant q-1 $.
	
	Further, we
	write $ i=i'+(q^2+q)l' $ with $ 0 \leqslant i' <q^2+q $. Let $  k_{\mu} =  k_{\mu}'-l' $ and $ j_\nu=j_\nu'-(q+1)l'  $ for $  \mu \geqslant 1$ and  $\nu \geqslant 0 $. The set 	$ 	\Omega_{s_0,s_1, \cdots, s_{q-1} ,t_1,\cdots, t_{q-1}, t_0}  $ in \eqref{eq:Ome_t2} becomes	
	\begin{align*}
	\Big\{& (i',l',j_0',j_1',\cdots,j_{q-1}',k_1',\cdots,k_{q-1}')
	~\Big|~ -s_0 \leqslant i'+(q^2+q)l', \,\,0 \leqslant i' <q^2+q ,\\
	&	~ k_{\mu}' = \left \lceil \frac{-i'-s_{\mu}}{q^2+q} \right \rceil  ~~\text{for} ~~ \mu =1,\cdots, q-1,\nonumber \\
	&-t_\nu \leqslant  qi' + (q^2+q) j_\nu' - (q+1) \sum_{\mu = 1}^{q-1} k_\mu'-(q+1)l' < -t_\nu + (q^2+q) \textup{ for } \nu =0, \cdots, q-1,   \\
	& i'+\sum_{\nu = 0}^{q-1} j_\nu' \leqslant 0
	\Big\}.
	\end{align*}
	The first inequality gives that  $ l' \geqslant k_0': =  \ceil{ \dfrac{-i'-s_0}{q^2+q} }  $. By taking $ l'=  k_0' +\iota' $ with $ \iota' \geqslant 0 $, we can rewrite $ 	\Omega_{s_0,s_1, \cdots, s_{q-1} ,t_1,\cdots, t_{q-1}, t_0}   $ as
	\begin{align*}
	\Big\{& (i',\iota',j_0',j_1',\cdots,j_{q-1}',k_0',k_1',\cdots,k_{q-1}')
	~\Big| \,\,0 \leqslant i' <q^2+q , \,\, \iota' \geqslant 0, \\
	&	~ k_{\mu}' = \left \lceil \frac{-i'-s_{\mu}}{q^2+q} \right \rceil  ~~\text{for} ~~ \mu =0,1,\cdots, q-1,\nonumber \\
	&-t_\nu \leqslant  qi' + (q^2+q) j_\nu' - (q+1) \sum_{\mu = 0}^{q-1} k_\mu'-(q+1)\iota'  < -t_\nu + (q^2+q) \textup{ for } \nu =0, \cdots, q-1,   \\
	&i'+\sum_{\nu = 0}^{q-1} j_\nu' \leqslant 0
	\Big\}.
	\end{align*}
	This means that the number of the lattice points also does not depend on the order of $s_{\mu}$ with $ 0 \leqslant \mu \leqslant q-1 $, which concludes the first assertion. 
	
	For the second assertion, the proof is similar to that of Lemma \ref{lem:set_inv}. The details are showed below. By taking 
	\begin{align*}
	I &:= -q^2 i - (q^2+q) \sum_{\nu=1}^{q-1} j_{\nu}-(q+1)\sum_{\mu=1}^{q-1} k_{\mu},\\
	K_{\mu}&:= i + \sum_{\nu=1}^{q-1} j_{\nu} + j_{\mu} \textup{ for } \mu=1,\cdots,q-1,\\
	J_{\nu} &:= qi+ (q+1) \sum_{a=1}^{q-1} j_{a} + \sum_{\mu=1}^{q-1} k_{\mu}+k_{\nu}  \textup{ for } \nu=1,\cdots,q-1,
	\end{align*}
	the set $ \Omega_{s_0,s_1, \cdots, s_{q-1} ,t_1,\cdots, t_{q-1} ,t_0} $ in
	\eqref{eq:Omega2_0} becomes
\begin{align*} 
\Big\{ \, &  (I, J_1,\cdots,J_{q-1} , K_1,\cdots,K_{q-1} )   \, \Big|\,    -t_0 \leqslant I ,  \nonumber \\
&-t_\mu \leqslant  I + (q^2+q) K_\mu < -t_\mu+ (q^2+q) \textup{ for } \mu =1, \cdots, q-1,   \nonumber \\
&-s_\nu \leqslant  qI + (q^2+q) J_\nu - (q+1) \sum_{\mu = 1}^{q-1} K_\mu < -s_\nu + (q^2+q) \textup{ for } \nu =1, \cdots, q-1,  \nonumber \\
&-s_0 \leqslant  -q^2 I - (q^2+q) \sum_{\nu = 1}^{q-1} J_\nu -(q+1) \sum_{\mu = 1}^{q-1}  K_\mu    \,\Big\}.  
\end{align*}
We find that this is just the set $ \Omega_{t_0,t_1, \cdots, t_{q-1} ,s_1,\cdots, s_{q-1} ,s_0} $. So the second assertion follows directly from the first assertion. This finishes the whole proof.
\end{proof}

\begin{lem}  \label{lem:Omes1}
	Suppose that $ 0 \leqslant s_\mu, t_\nu < q^2+q $. Let $ s_1=\min_{1\leqslant \mu \leqslant q-1}\{s_{\mu}\} $, $ \widehat r = r-(q^2+q+1)s_1$, $ \widehat s_{\mu} =  s_{\mu}-s_1$, $  t_{\nu} + s_1 =\widehat t_{\nu} + (q^2+q)l_{\nu}$ and  $ \widehat u = u+ (q^2+q+1)s_1+\sum_{\nu=1}^{q-1} l_{\nu} $, then we have
	\begin{align*} 
	\# \Omega_{r ,s_1,s_2,\cdots,s_{q-1}, t_1, \cdots, t_{q-1}, u }
	& = \# \Omega_{\widehat r ,0, \widehat s_2,\cdots,\widehat s_{q-1}, \widehat t_1, \cdots, \widehat t_{q-1}, \widehat u } , \\
	r+ \sum_{\mu=1}^{q-1} s_\mu +\sum_{\nu=1}^{q-1} t_\nu + u &=
	\widehat r+ \sum_{\mu=2}^{q-1} \widehat s_\mu +\sum_{\nu=1}^{q-1} \widehat t_\nu +\widehat u.
	\end{align*}
\end{lem}
\begin{proof}
	Note that the set $ \Omega_{r,s_1, \cdots, s_{q-1} ,t_1,\cdots, t_{q-1}, u} $ is defined by \eqref{eq:Omega2_0}. 
	By setting $ I := i+(q^2+q+1)s_1 $, $ K_{\mu} := k_{\mu} -s_1 $ and $ J_{\nu} := j_{\nu}-(q+1)s_1 -l_{\nu} $, we find that $  \Omega_{r,s_1, \cdots, s_{q-1} ,t_1,\cdots, t_{q-1}, u} $ becomes
	\begin{align*}
	\Big\{ \, &  (I, J_1,\cdots,J_{q-1} , K_1,\cdots,K_{q-1} )   \, \Big|\,    
	-\widehat r \leqslant I ,  \\
	& \phantom{- s} 0 \leqslant  I + (q^2+q) K_1 <  (q^2+q),\\
	& -\widehat s_\mu \leqslant I + (q^2+q) K_\mu < -\widehat s_\mu+ (q^2+q)  \textup{ for } \mu =2, \cdots, q-1,   \\
	&-\widehat t_\nu \leqslant  qI + (q^2+q) J_\nu - (q+1) \sum_{\mu = 1}^{q-1} K_\mu < -\widehat t_\nu + (q^2+q) \textup{ for } \nu =1, \cdots, q-1,   \\
	&-\widehat  u \leqslant  -q^2 I - (q^2+q) \sum_{\nu = 1}^{q-1} J_\nu -(q+1) \sum_{\mu = 1}^{q-1}  K_\mu    \,\Big\},  
	\end{align*}
	which is exactly the set  $ \Omega_{\widehat r ,0, \widehat s_2,\cdots,\widehat s_{q-1}, \widehat t_1, \cdots, \widehat t_{q-1}, \widehat u } $, giving the desired conclusion. This finishes the proof.	 
\end{proof}

\begin{lem} \label{lem:s}
	Let $r^*: = q^3 - q $, $u ^{*}:=  q^3+ q^2 $. Assume that  $ 0 \leqslant s_\mu, t_\nu < q^2+q $ and $ u \geqslant u^* $, 
	then we have
	\[\# \Omega_{r^*, 0 , s_2,  \cdots, s_{q-1}, t_1, \cdots, t_{q-1}, u  } = \# \Omega_{r^*, 0 , 0, \cdots, 0, t_1,  \cdots, t_{q-1}, u }  + \sum_{\mu=2}^{q-1} s_\mu . \]  
\end{lem}
\begin{proof}
	The proof is proceeded as that of Lemma \ref{lem:Omega_main}. We start the proof by fixing the index $k_1 $. Then we define 
	\[  \Theta_{k}: =  \Big\{\, (i, j_1, \cdots, j_{q-1},  k_1 = k, k_2, \cdots, k_{q-1}) \in \Omega_{r^* ,0,s_2, \cdots, s_{q-1} ,t_1,\cdots, t_{q-1}, u} \,  \Big\}, \] 
	where $ k=\ceil{\dfrac{-i}{q^2+q}} \leqslant q-1 $.
	So we obtain 
	\begin{align}\label{eq:Thetak}
	\# \Omega_{r^*, 0 , s_2,  \cdots, s_{q-1}, t_1,   \cdots, t_{q-1}, u } =\sum _{k=- \infty }^{q-1}  \# \Theta_k.  
	\end{align}
	
	For convenience we always assume the  subscripts $ \nu = 1 , \cdots, q -1 $, $ \mu =2, \cdots , q-1 $ and $ \kappa = 3 , \cdots , q-1 $.
	
	Let $ \widetilde{i} := i + (q^2+q) k  $,   $ \widetilde{k}_\mu =  k_\mu - k  $ and $ \widetilde{j}_\nu : = j_\nu -(q+1) k   $. We find that $ \Theta_k$ is equivalent to 
	\begin{align}
	\Phi_{k}: =  \Big\{ \,& (\widetilde{i}, \widetilde{j}_1,\cdots,\widetilde{j}_{q-1}, \widetilde{k}_2,\cdots,\widetilde{k}_{q-1} ) \,\Big|\,   0  \leqslant \widetilde{i} < q^2+q , \nonumber \\
	&-s_2\leqslant \widetilde{i} + (q^2+q) \widetilde{k}_2< -s_2+ (q^2+q) , \label{eq:wide_Phi_2} \tag{$G_2 $}  \\
	&-s_\kappa \leqslant \widetilde{i} + (q^2+q) \widetilde{k}_\kappa < -s_\kappa+ (q^2+q) \textup{ for } 3\leqslant \kappa \leqslant q-1,  \nonumber \\
	& -t_\nu \leqslant  q\widetilde{i} + (q^2+q) \widetilde{j}_\nu  +  (q+1)k - (q+1) \sum_{\mu=2}^{q-1}   \widetilde{k}_\mu < -t_\nu + (q^2+q) \textup{ for } 1\leqslant \nu \leqslant q-1, \nonumber \\
	& -u \leqslant  -q^2 \widetilde{i}    - (q^2+q) \sum_{\nu=1}^{q-1}  \widetilde{j}_\nu   -(q+1) \sum_{\mu=2}^{q-1}  \widetilde{k} _\mu  +(q+1)k 	    \, \Big\}. \nonumber
	\end{align}
	We decompose the set  $   \Phi_{k} $ by splitting Condition (\ref{eq:wide_Phi_2}), namely $ \Phi_{k}  =  \Phi_{k}^{(1)} \cup  \Phi_{k}^{(2)} $, where
	\begin{align*}
	\Phi_{k}^{(1)}: =  \Big\{ \,  &(\widetilde{i}, \widetilde{j}_1,\cdots,\widetilde{j}_{q-1}, \widetilde{k}_2,\cdots,\widetilde{k}_{q-1} ) \,\Big|\,   0  \leqslant \widetilde{i} < q^2+q , \\
	&-s_2 \leqslant \widetilde{i} + (q^2+q) \widetilde{k}_2< 0 ,  \\
	&-s_\kappa \leqslant \widetilde{i} + (q^2+q) \widetilde{k}_\kappa < -s_\kappa+ (q^2+q) \textup{ for } 3\leqslant \kappa \leqslant q-1 ,  \\
	&-t_\nu \leqslant  q\widetilde{i} + (q^2+q) \widetilde{j}_\nu  +  (q+1)k - (q+1) \sum_{\mu=2}^{q-1} \widetilde{k}_\mu < -t_\nu + (q^2+q) \textup{ for } 1\leqslant \nu \leqslant q-1, \\
	& -u \leqslant  -q^2 \widetilde{i}    - (q^2+q) \sum_{\nu=1}^{q-1}  \widetilde{j}_\nu   -(q+1) \sum_{\mu=2}^{q-1} \widetilde{k} _\mu  +(q+1)k \, \Big\}, 
	\end{align*}
	and
	\begin{align*}
	\Phi_{k}^{(2)}: =  \Big\{ \, &(\widetilde{i}, \widetilde{j}_1,\cdots,\widetilde{j}_{q-1}, \widetilde{k}_2,\cdots,\widetilde{k}_{q-1} )   \,\Big|\,  0  \leqslant \widetilde{i} < q^2+q , \\
	&\phantom{-s}0 \leqslant \widetilde{i} + (q^2+q) \widetilde{k}_2< -s_2+ (q^2+q) ,   \\
	& -s_\kappa \leqslant \widetilde{i} + (q^2+q) \widetilde{k}_\kappa < -s_\kappa+ (q^2+q) \textup{ for } 3\leqslant \kappa \leqslant q-1,  \\
	& -t_\nu \leqslant  q\widetilde{i} + (q^2+q) \widetilde{j}_\nu  +  (q+1)k - (q+1) \sum_{\mu=2}^{q-1} \widetilde{k}_\mu < -t_\nu + (q^2+q) \textup{ for } 1\leqslant \nu \leqslant q-1, \\
	& -u 	  \leqslant  -q^2 \widetilde{i}    - (q^2+q) \sum_{\nu=1}^{q-1}  \widetilde{j}_\nu   -(q+1) \sum_{\mu=2}^{q-1} \widetilde{k} _\mu  +(q+1)k  \, \Big\}.
	\end{align*}
	
	Let $\widehat k_{2} =\widetilde{k}_{2} +1 $, then $ \Phi_k^{(1)}$ becomes 
	\begin{align*}
	\widehat{\Phi}_{k}^{(1)}: = \Big\{\, &(\widetilde{i}, \widetilde{j}_1,\cdots,\widetilde{j}_{q-1}, \widehat{k}_2,\widetilde{k}_3,\cdots,\widetilde{k}_{q-1}  ) \,\Big|\, 0  \leqslant \widetilde{i} < q^2+q ,\\
	&\phantom{-}(q^2  +q )-s_2  \leqslant \widetilde{i} + (q^2+q) \widehat{k}_2<  q^2+q ,\\
	&-s_\kappa \leqslant \widetilde{i} + (q^2+q) \widetilde{k}_\kappa < -s_\kappa+ (q^2+q) \textup{ for } 3\leqslant \kappa \leqslant q-1,\\
	&-t_\nu \leqslant  q\widetilde{i} + (q^2+q) \widetilde{j}_\nu  +  (q+1)(k+1) - (q+1) \widehat{k}_2 - (q+1) \sum_{\kappa=3}^{q-1}  \widetilde{k}_\kappa < -t_\nu + (q^2+q)\\
	& \qquad \qquad \qquad \qquad \qquad \qquad \qquad \qquad  \qquad \qquad \qquad \qquad  \textup{  for  } 1\leqslant \nu \leqslant q-1,\\
	&-u \leqslant  -q^2 \widetilde{i}    - (q^2+q) \sum_{\nu=1}^{q-1} \widetilde{j}_\nu  -(q+1)\widehat{k}_2  -(q+1) \sum_{\kappa=3}^{q-1}  \widetilde{k} _\kappa  +(q+1)(k+1)	~  \Big\}.
	\end{align*}
	Let us define the lattice point set
	\begin{align*}
	{\Psi}_{k} : = \Big\{\, &( \widetilde{i},\widetilde{j}_1,\cdots,\widetilde{j}_{q-1}, \widetilde{k}_2,\cdots,\widetilde{k}_{q-1} )  \,\Big|\, 0  \leqslant  \widetilde{i} < q^2+q ,\\
	&\phantom{-s}0   \leqslant \widetilde{i} + (q^2+q) \widetilde{k}_2<  q^2+q ,\\
	& -s_\kappa \leqslant \widetilde{i} + (q^2+q) \widetilde{k}_\kappa < -s_\kappa+ (q^2+q) \textup{ for } 3\leqslant \kappa \leqslant q-1,\\
	&-t_\nu \leqslant  q\widetilde{i} + (q^2+q) \widetilde{j}_\nu  +  (q+1)k   - (q+1) \sum_{\mu=2}^{q-1}  \widetilde{k}_\mu < -t_\nu + (q^2+q) \textup{ for } 1\leqslant \nu \leqslant q-1,\\
	&-u \leqslant  -q^2 \widetilde{i}    - (q^2+q) \sum_{\nu=1}^{q-1} \widetilde{j}_\nu     -(q+1) \sum_{\mu=2}^{q-1}  \widetilde{k} _\mu  +(q+1)k	~  \Big\}.
	\end{align*}
	Note that
	\begin{equation*}
	\widehat \Phi_{k}^{(1)} \cup\Phi_{k+1}^{(2)} = \Psi _{k+1}.
	\end{equation*}
	Then we have
	\begin{equation} \label{eq:widehat_Phi}
	\sum_{k=-\infty}^{q-1} \#  \Phi _{k} =\sum_{k=-\infty}^{q-1} \#  \Phi^{(1)} _{k}+\sum_{k=-\infty}^{q-1} \#  \Phi^{(2)} _{k} = \sum_{k=-\infty}^{q-1} \#  \Psi _{k}+ \# \Phi^{(1)}_{q-1}  .
	\end{equation}
	By a similar argument as we have done in the proof of Lemma \ref{lem:Omega_large}, we claim that 
	\[  \# \Phi_{q-1}^{(1)}= s_2  . \]
	If we combine Equations \eqref{eq:Thetak} and \eqref{eq:widehat_Phi}, then
	\begin{align}\label{eq:Omega_s2}
	\# \Omega_{r^*, 0 , s_2,s_3, \cdots, s_{q-1}, t_1, t_2, \cdots, t_{q-1}, u } = \# \Omega_{r^*, 0 , 0,s_3 \cdots, s_{q-1}, t_1, t_2, \cdots, t_{q-1},u }  + s_2.  
	\end{align}  
	Applying Lemma \ref{lem:set_inv2} and Equation 
	\eqref{eq:Omega_s2}, we get
	\[\# \Omega_{r^*, 0 , s_2,s_3, \cdots, s_{q-1}, t_1, t_2, \cdots, t_{q-1}, u} = \# \Omega_{r^*, 0 , 0, \cdots, 0, t_1, t_2, \cdots, t_{q-1},u }  + \sum s_\mu  .\]
	This finishes the proof.   
\end{proof}

\begin{lem} \label{lem:Omega_large2}
	Let $r^*: = q^3 - q $, $u ^{*}:=  q^3+ q^2 $. If $ r \geqslant r^* $, $ u \geqslant u^* $ and $0 \leqslant s_\mu, t_\nu < q^2+q $ for $ 1\leqslant \mu,\nu \leqslant q-1 $, we have
	\[ \# \Omega_{r , s_1, \cdots, s_{q-1}, t_1, t_2,\cdots, t_{q-1}, u } =\# \Omega_{r^* , s_1, \cdots, s_{q-1}, t_1, t_2,\cdots, t_{q-1}, u^* } +( r-r^*) + (u-u^* ). \]
\end{lem} 
\begin{proof}
	The proof is similar to that of Lemma \ref{lem:Omega_large}.
\end{proof}

\begin{lem}\label{lem:phi2}
	Let 
	\begin{align*}
	\Gamma_{b}^{(1)}: = \Big\{\, ( {i},  {j}) \,\Big|\, 0   \leqslant  {i} < q^2+q ,
	\,\,
	- t_1 &\leqslant  q {i}+ (q^2+q) \ {j}  + (q+1)b  < 0 	 \,\Big\}.
	\end{align*}	
	Then $ \sum_{b=0}^{q-1} \# \Gamma_b ^{(1)}= qt_1 $.
\end{lem}
\begin{proof}
	The proof is similar to that of Lemma \ref{lem:Phi_b}. Let  
	\begin{align*}
	{H}_{t}: = \Big\{\, ( {i},  {j})\,\Big|\, 0   \leqslant  {i} < q^2+q , \,\,	  q {i}+ (q^2+q)  {j}  + qb=-t    \,\Big\}.
	\end{align*}
	Then one verifies that
	\[  	\# H_t = \begin{cases}
	q  & \textup{ for $ q \mid t $}, \\
	0  & \textup{ for $ q \nmid t $} .
	\end{cases} \]
	It follows that
	\begin{align*}
	\sum_{b=0}^{q-1} \# \Gamma_b ^{(1)}& =  \sum_{b=0}^{q-1} \sum_{t=b+1}^{t_1+b } \# H_t  
	 =  \sum_{b=0}^{q-1} \floor{\frac{t_1+b }{q}}
	=q t_1,
	\end{align*}
	completing the proof of this lemma.
\end{proof}

\begin{lem}\label{lem:t}
	Set $r^*: = q^3-q$, $u ^{*}:=  q^3+q^2 $.
	Assume that  $ 0\leqslant t_\nu <q^2+q $ for $ 1\leqslant \nu \leqslant q-1 $, 
	then we have
	\[ \# \Omega_{r^*, 0, \cdots, 0, \,t_1,  \cdots, t_{q-1}, u^* } = \# \Omega_{r^*, 0, \cdots, 0,\, 0,  \cdots, 0, u^* } + \sum_{\nu =1 }^{q-1} t_\nu.  \]
\end{lem}
\begin{proof}
	The proof is similar to that of Lemma \ref{lem:s} and we give the details here.
	We consider the set $ \Omega_{r^*, 0, \cdots, 0, t_1,  \cdots, t_{q-1}, u  } $, where $u = u^{*} + (q^2+q )l$ for $ 0 \leqslant l < q  $. Define
	\[  \Gamma_{k} :=  \Big\{\, (i, j_1, \cdots, j_{q-1}) \,\Big| \, (i, j_1, \cdots, j_{q-1},  k_1 =    \cdots= k_{q-1}=k) \in \Omega_{r^* ,0, \cdots, 0 ,t_1, \cdots, t_{q-1}, u} \,  \Big\}, \] 
	where $ k=\ceil{\dfrac{-i}{q^2+q}} $. Plugging $ \widetilde{i}: = i + (q^2+q) k  $  and $ \widetilde{j}_\nu : = j_\nu -(q+1) k   $ into $ \Gamma_{k} $ gives that 
	\begin{align*}
	\Gamma_{k} = \Big\{ \, &( \widetilde{i}, \widetilde{j}_1,\cdots,\widetilde{j}_{q-1} ) \, \Big| \, 0  \leqslant \widetilde{i} < q^2+q ,\\
	&-t_\nu \leqslant  q\widetilde{i} + (q^2+q) \widetilde{j}_\nu  +  (q+1)k   < -t_\nu + (q^2+q) \textup{ for } 1 \leqslant \nu \leqslant q-1,\\
	&-u \leqslant  -q^2 \widetilde{i}    - (q^2+q) \sum_{\nu = 1}^{q-1} \widetilde{j}_\nu      +(q+1)k 	\,\Big\}.
	\end{align*}
	Write $ k=-qa  +b $ with $a, b \in \mathbb{Z}$ and $0 \leqslant b <q  $. Let $\overline{i} := \widetilde{i}$ and $ \overline{j}_{\nu}  : =\widetilde{ j}_{\nu} -a $. It then follows that $ \Gamma_{k} $ becomes
	\begin{align*}
	{\Gamma} _{a,b,l}: = \Big\{\, & (\overline{i}, \overline{j}_1, \cdots,\overline{j}_{q-1} )\,\Big|\, 0 \leqslant \overline{i} < q^2+q ,\\
	&- t_1 \leqslant  q\overline{i}+ (q^2+q) \overline{j}_1 + (q+1) b   <  -t_1  + (q^2+q) ,\\
	&-t_\nu \leqslant  q\overline{i}+ (q^2+q) \overline{j}_\nu  +  (q+1) b < -t_\nu + (q^2+q) \textup{ for } 2 \leqslant \nu \leqslant q-1,\\
	&-(u^{*} + (q^2+q )l ) \leqslant  -q^2 \overline{i}    - (q^2+q) \sum_{\nu = 1}^{q-1} \overline{j}_\nu     +(q+1)b - (q^3+q^2 )a	\, \Big\}.
	\end{align*}
	The set  $   \Gamma_{a,b,l} $ has a partition $ \Gamma_{a,b,l}  =  \Gamma_{a,b,l}^{(1)} \cup  \Gamma_{a,b,l}^{(2)} $, where
	\begin{align*}
	{\Gamma} _{a,b,l}^{(1)}: = \Big\{ \, & (\overline{i}, \overline{j}_1, \cdots,\overline{j}_{q-1} )\,\Big|\, 0  \leqslant \overline{i} < q^2+q ,\\
	&- t_1 \leqslant  q\overline{i}+ (q^2+q) \overline{j}_\nu + (q+1)b  < 0 ,\\
	&-t_\nu \leqslant  q\overline{i}+ (q^2+q) \overline{j}_\nu  +  (q+1) b < -t_\nu + (q^2+q) \textup{ for } 2 \leqslant \nu \leqslant q-1,\\
	&-(u^{*} + (q^2+q )l )\leqslant  -q^2 \overline{i}    - (q^2+q) \sum_{\nu = 1}^{q-1} \overline{j}_\nu     +(q+1)b - (q^3+q^2 )a	\, \Big\},
	\end{align*}
	and
	\begin{align*}
	{\Gamma} _{a,b,l}^{(2)}: = \Big\{ \, & (\overline{i}, \overline{j}_1, \cdots,\overline{j}_{q-1} )\,\Big|\, 0 \leqslant \overline{i} < q^2+q ,\\
	& \phantom{-s} 0 \leqslant  q\overline{i}+ (q^2+q) \overline{j}_1 +  (q+1) b    < -t_1 + (q^2+q) ,\\
	&-t_\nu \leqslant  q\overline{i}+ (q^2+q) \overline{j}_\nu  +  (q+1) b < -t_\nu + (q^2+q) \textup{ for } 2 \leqslant \nu \leqslant q-1,\\
	&-(u^{*} + (q^2+q )l) \leqslant  -q^2 \overline{i}    - (q^2+q) \sum_{\nu = 1}^{q-1} \overline{j}_\nu     +(q+1)b - (q^3+q^2 )a	\, \Big\}.
	\end{align*}
	
	Let $  \widehat{j}_1:= \widetilde {j}_1 +1 $, then $ {\Gamma} _{a,b,l}^{(1)} $ is equivalent to
	\begin{align*}
	\widehat{\Gamma} _{a,b,l}^{(1)}: = \Big\{ \, & (\overline{i}, \widehat{j}_1, \overline{j}_2, \cdots,\overline{j}_{q-1} )\,\Big|\, 0 \leqslant \overline{i} < q^2+q ,\\
	&-t_1+(q^2+q) \leqslant  q\overline{i}+ (q^2+q) \widehat{j}_1  +(q+1)b  < q^2+q,\\
	&-t_\nu \leqslant  q\overline{i}+ (q^2+q) \overline{j}_\nu  +  (q+1) b < -t_\nu + (q^2+q) \textup{ for } 2 \leqslant \nu \leqslant q-1,\\
	&-(u^{*} + (q^2+q )(l+1)) \leqslant  -q^2 \overline{i}   - (q^2+q) \widehat{j}_1 - (q^2+q) \sum_{\nu =2}^{q-1} \overline{j}_\nu     +(q+1)b - (q^3+q^2 )a	\, \Big\}.
	\end{align*}
	We define a lattice point set
	\begin{align*}
	{\Psi} _{a,b,l} : = \Big\{ \, & (\overline{i}, \overline{j}_1, \cdots,\overline{j}_{q-1} )\,\Big|\, 0 \leqslant \overline{i} < q^2+q ,\\
	&\phantom{-s} 0 \leqslant  q\overline{i}+ (q^2+q) \overline{j}_1 +(q+1)b   < q^2+q,\\
	&-t_\nu \leqslant  q\overline{i}+ (q^2+q) \overline{j}_\nu  +  (q+1) b < -t_\nu + (q^2+q) \textup{ for } 2 \leqslant \nu \leqslant q-1,\\
	&-(u^{*} + (q^2+q )l) \leqslant  -q^2 \overline{i}    - (q^2+q) \sum_{\nu = 1}^{q-1} \overline{j}_\nu     +(q+1)b - (q^3+q^2 )a	\,\Big\}.
	\end{align*} 	
	Then the sets $ {\Psi} _{a,b,l} $ are splited into the following forms:
	\begin{align*}
	\Psi_{a, b,1}& =\widehat \Gamma_{a,b,0}^{(1)} \cup 	\Gamma_{a,b,1}^{(2)}, \\
	\Psi_{a, b,2}& =\widehat \Gamma_{a,b,1}^{(1)} \cup 	\Gamma_{a,b,2}^{(2)}, \\
	& \cdots, \\
	\Psi_{a, b,q-1}& =\widehat \Gamma_{a,b,q-2}^{(1)} \cup 	\Gamma_{a,b,q-1}^{(2)}, \\
	\Psi_{a , b,0}& =\widehat \Gamma_{a+1,b,q-1}^{(1)} \cup 	\Gamma_{a ,b,0}^{(2)}. 	 	  	 	  	 	 
	\end{align*} 
	Consequently we have 
	\begin{align*} \sum _ { l=0}^{q-1 } \# \Psi_{a,b,l} &= \# \widehat\Gamma_{a+1,b,q-1}^{(1)}+  \sum _ { l=0}^{q-2 } \#\widehat \Gamma_{a,b,l}^{(1)}+ \sum _ { l=0}^{q-1 } \# \Gamma_{a,b,l}^{(2)},\\   
	\sum _ { l=0}^{q-1 }\sum_{a=0}^{\infty}\# \Psi_{a,b,l}& =\sum_{a=1}^{\infty} \# \widehat\Gamma_{a,b,q-1}^{(1)}+  \sum _ { l=0}^{q-2 }\sum_{a=0}^{\infty} \#\widehat \Gamma_{a,b,l}^{(1)}+ \sum _ { l=0}^{q-1 }\sum_{a=0}^{\infty}\# \Gamma_{a,b,l}^{(2)}.
	\end{align*}
	The last equation leads to
	\begin{align}
	\sum _ { b=0}^{q-1 }\sum _ { l=0}^{q-1 }\sum_{a=0}^{\infty}\# \Gamma_{a,b,l}  &= \sum _ { b=0}^{q-1 }\sum _ { l=0}^{q-1 }\sum_{a=0}^{\infty}\# \Gamma_{a,b,l}^{(1)}+\sum _ { b=0}^{q-1 }\sum _ { l=0}^{q-1 }\sum_{a=0}^{\infty}\# \Gamma_{a,b,l}^{(2)}\nonumber \\
	&= 
	\sum _ { b=0}^{q-1 }\sum _ { l=0}^{q-1 }\sum_{a=0}^{\infty}\# \Gamma_{a,b,l}^{(1)}
	+\sum _ { b=0}^{q-1 }
	\Big(     \sum _ { l=0}^{q-1 }\sum_{a=0}^{\infty}\# \Psi_{a,b,l}-\sum_{a=1}^{\infty} \# \widehat\Gamma_{a,b,q-1}^{(1)}-  \sum _ { l=0}^{q-2 }\sum_{a=0}^{\infty} \#\widehat \Gamma_{a,b,l}^{(1)} \Big)\nonumber \\     
	&= \sum _ { b=0}^{q-1 }\sum _ { l=0}^{q-1 }\sum_{a=0}^{\infty}\# \Psi_{a,b,l} +  \sum _ { b=0}^{q-1 } \# \Gamma_{0,b,q-1}^{(1)} . \label{eq:Phi1}
	\end{align}
	Set $ N(t_1): = \# \Omega_{r^*, 0, \cdots, 0, t_1, t_2,\cdots, t_{q-1}, u^* } $.
	It follows from the definitions of $ \Gamma_{a,b,l} $ and $ \Psi_{a,b,l} $ that 
	\begin{align}
	\sum _ { b=0}^{q-1 }\sum _ { l=0}^{q-1 }\sum_{a=0}^{\infty}\# \Gamma_{a,b,l}& = \sum _ { l=0}^{q-1 } \left( N(t_1) + (q^2+q) l \right) = q N(t_1) + (q^4-q^2)/2, \label{eq:phi1}\\ 
	\sum _ { b=0}^{q-1 }\sum _ { l=0}^{q-1 }\sum_{a=0}^{\infty}\# \Psi_{a,b,l} &= \sum _ { l=0}^{q-1 } \left( N(0) + (q^2+q) l \right) = q N(0) + (q^4-q^2)/2.\label{eq:phi2}
	\end{align}
	By Lemma \ref{lem:phi2}, 
	\begin{align}\label{eq:phi3}
	\sum _ { b=0}^{q-1 } \# \Gamma_{0,b,q-1}^{(1)} = q t_1 .
	\end{align}
	Therefore we obtain from Equations \eqref{eq:Phi1}, \eqref{eq:phi1}, \eqref{eq:phi2} and \eqref{eq:phi3} that
	$ N(t_1) = N(0 ) + t_1  $, i.e.,   
	\begin{equation}\label{eq:Omet1}
	\# \Omega_{r^*, 0, \cdots, 0, t_1, t_2,\cdots, t_{q-1}, u^* } = \# \Omega_{r^*, 0, \cdots, 0, 0, t_2,\cdots, t_{q-1}, u^* } + t_1 .
	\end{equation}
	Applying Lemma \ref{lem:set_inv2} and Equation \eqref{eq:Omet1} gives that
	\[ \# \Omega_{r^*, 0, \cdots, 0, t_1, t_2,\cdots, t_{q-1}, u^* } = \# \Omega_{r^*, 0, \cdots, 0, 0, 0,\cdots, 0, u^* } + \sum_{\nu=1}^{q-1} t_\nu . \]
	This finishes the proof.  
\end{proof}

\begin{lem}\label{lem:Omeru}
	Set $r^*: = q^3-q$, $u ^{*}:=  q^3+q^2 $. Then
	$ \# \Omega_{r^*,0, \cdots, 0 ,0,\cdots, 0, u^*} =1-g^{(3)} +r^* + u^* $.
\end{lem}
\begin{proof}
	It follows from the definition that
	\begin{align*}
	\Omega_{r^*,0, \cdots, 0 ,0,\cdots, 0, u^*}: = \Big\{ \,&  (i, j_1,\cdots,j_{q-1} , k_1,\cdots,k_{q-1}  )    \, \Big| \, -r^* \leqslant i ,\\
	&\phantom{-} 0 \leqslant  i + (q^2+q) k_\mu < q^2+q  \textup{ for } 1 \leqslant \mu \leqslant q-1,\\
	&\phantom{-} 0 \leqslant  qi + (q^2+q) j_\nu - (q+1) \sum_{\mu = 1}^{q-1} k_\mu <  q^2+q  \textup{ for } 1 \leqslant \nu \leqslant q-1,\\
	&-u^* \leqslant  -q^2 i - (q^2+q) \sum_{\nu = 1}^{q-1} j_\nu -(q+1) \sum_{\mu = 1}^{q-1}  k_\mu   \,\Big\}.
	\end{align*}
	Since $ k_1=k_2=\cdots =k_{q-1} $ and $ j_1=j_2=\cdots =j_{q-1} $,  we rewrite the set $ \Omega_{r^*,0, \cdots, 0 ,0,\cdots, 0, u^*} $ as 
	\begin{align*}
	\Omega_{r^*,0, \cdots, 0 ,0,\cdots, 0, u^*} = \Big\{\, (i, j  , k  )\, \Big|\, -r^* & \leqslant i ,\\
	0 &\leqslant  i + (q^2+q) k < q^2+q  ,\\
	0 &\leqslant  qi + (q^2+q) j  - (q^2-1)  k <  q^2+q  ,\\
	-u^* &\leqslant  -q^2 i - (q^3-q) j -(q^2-1)   k    	\,\Big\}.
	\end{align*}
	Let $ j'= j-k $. Then the set $ \Omega_{r^*,0, \cdots, 0 ,0,\cdots, 0, u^*}  $ becomes
	\begin{align*}
	\Big\{\, (i, j'  , k  ) \,\Big| \, -r^* & \leqslant i ,\\
	0 &\leqslant  i + (q^2+q) k < q^2+q  ,\\
	0 &\leqslant  qi + (q^2+q) j' +(q+1) k <  q^2+q  ,\\
	-u^* &\leqslant  -q^2 i - (q^3-q) j' -(q^3 + q^2-q-1)   k    	\,\Big\},
	\end{align*}
	which is exactly the set $ \Omega_{r^*,\,0,\,0,\,u^*}$ of \eqref{eq:O1}. 
	So we have from Lemma \ref{lem:Omega_main} that $ \# \Omega_{r^*,0, \cdots, 0 ,0,\cdots, 0, u^*} =1-g^{(3)}  +r^* + u^* $.
\end{proof}


The proof of Proposition \ref{prop:Omega2} is given as follows.
\begin{proof}[Proof of Proposition \ref{prop:Omega2}]
	When $ r \geqslant R $, the cardinality of $ \Omega_{r ,s_1,\cdots,s_{q-1},t_1,\cdots,t_{q-1} , u } $ is given as follows: 
	\begin{align*} 
	\# \Omega_{r ,s_1,\cdots,s_{q-1},t_1,\cdots,t_{q-1} , u }&  = \# \Omega_{\widehat r,0 ,\widehat s_2,\cdots,\widehat s_{q-1} ,\widehat t_1,\cdots, \widehat t_{q-1}, \widehat u} \quad \quad \textup{(by Lemmas \ref{lem:set_inv2} and \ref{lem:Omes1})}\\
	& =  \# \Omega_{r^* ,0 ,\widehat s_2,\cdots,\widehat s_{q-1} , \widehat t_1,\cdots, \widehat t_{q-1} , u^* } + (\widehat r -r^* )  +(\widehat u - u^*) \quad \quad \textup{(by Lemma \ref{lem:Omega_large2})} \\
	& =  \# \Omega_{  r^*  ,0,0,\cdots,0, \widehat t_1,\cdots, \widehat t_{q-1} ,   u^*  } + (\widehat r -r^* ) +  \sum_{\nu=2}^{q-1 } \widehat s_\nu  +(\widehat u - u^*) \quad \quad \textup{(by Lemma \ref{lem:s})}  \\			
	& = \# \Omega_{r^*, 0, 0,\cdots,  0, 0,\cdots, 0, u^* } + (\widehat r -r^* ) +  \sum_{\mu=2}^{q-1 } \widehat s_\mu + \sum_{\nu =1 }^{q-1} \widehat t_\nu +(\widehat u - u^*) \quad \quad \textup{(by Lemma \ref{lem:t})}\\
	& = 1- g^{(3)}+ r^* + u^* + (\widehat r -r^* ) + \sum_{\mu=2}^{q-1 } \widehat s_\mu + \sum_{\nu =1 }^{q-1} \widehat t_\nu+(\widehat u - u^*)  \quad \quad \textup{(by Lemma \ref{lem:Omeru})} \\
	& = 1- g^{(3)}+  \widehat r  + \sum_{\mu=2}^{q-1 } \widehat s_\mu + \sum_{\nu =1 }^{q-1} \widehat t_\nu + \widehat u   \\ 
	& = 1- g^{(3)}+   r  + \sum_{\mu=1}^{q-1 }   s_\mu + \sum_{\nu =1 }^{q-1}   t_\nu +  u .  \quad \quad \textup{(by Lemma \ref{lem:Omes1})}  
	\end{align*}
	This finishes the proof of Proposition \ref{prop:Omega2}.	
\end{proof}

\section{Weierstrass semigroups and pure gaps  }\label{sec:Weiersemipureg}
In this section, we calculate the Weierstrass semigroups and the pure gap sets at the places $  Q^{(3)} $,  $ S_{0, \mu }^{(3)} $, $ S_{1, \nu }^{(3)} $, $ P^{(3)} $, for $ 1 \leqslant \mu,\nu \leqslant q-1 $. 

We briefly introduce some preliminary notation and results before we begin. Let $ Q_1,\cdots, Q_k $ be distinct rational places of a function field $ F $. The Weierstrass semigroup
$ H(Q_1,\cdots, Q_k) $ is defined by
\[
\Big\{(s_1,\cdots, s_k)\in \mathbb{N}_0^k\,\,\Big| \,\,\exists f\in F \text{ with } \Div_{\infty}(f)=\sum_{i=1}^k s_i Q_i  \Big\},
\]
and the Weierstrass gap set $  G(Q_1,\cdots, Q_k)  $ is defined by $ \mathbb{N}_0^k \backslash H(Q_1,\cdots, Q_l) $, where $  \Div_{\infty}(f) $ is the pole divisor of $ f $ and $ \mathbb{N}_0= \mathbb{N} \cup \{0\}$ with $ \mathbb{N} $ denotes the set of positive integers. For more details, we refer  \cite{matthews2004weierstrass}.

Homma and Kim \cite{Homma2001Goppa} introduced the concept of pure gap set with respect to a pair of rational places. It was generalized by
Carvalho and Torres \cite{carvalho2005goppa} to several rational places, denoted by $ G_0(Q_1,\cdots, Q_k) $, which is given by
\begin{align*}
\Big\{&(s_1,\cdots,s_k)\in \mathbb{N}^k\,\,\Big| \,\,\ell(G) = \ell(G -Q_j ) \text{ for all }1\leqslant j \leqslant k, \text{ where }G=\sum_{i=1}^k s_iQ_i \Big\}.
\end{align*}
In \cite{carvalho2005goppa}, they also proved that $ (s_1,\cdots,s_k)  $ is a pure gap at $ (Q_1,\cdots, Q_k) $ if and only if
\begin{align*}
\ell(s_1Q_1+\cdots+s_k Q_k)=\ell((s_1-1)Q_1+\cdots+(s_k-1) Q_k).
\end{align*}

There is a useful way to calculate the Weierstrass semigroups, which can be regarded as an easy generalization of Lemma 2.1 due to Kim \cite{kim}.
\begin{lem}[\cite{carvalho2005goppa}, Lemma 2.2]
	\label{lem:Weierstsemi}
	For $ k $ distinct rational places $ Q_1,\cdots,Q_k $, the set $ H(Q_1,\cdots,Q_k) $ is given by
	\begin{align*}
	\Big\{&(s_1,\cdots,s_k)\in \mathbb{N}_0^k~\Big|~\ell(G) \neq \ell(G -Q_j ) \text{ for } 1\leqslant j \leqslant k, \text{ where } G=\sum_{i=1}^k s_iQ_i \Big\}.
	\end{align*}

\end{lem}

For the purpose of calculating Weierstrass semigroups and pure gaps, we require auxiliary results which are described below.

\begin{lem}\label{lem:omegawith1}
	The following assertions hold.
	\begin{enumerate}
		\item
		$ \#\Omega_{s_0,s_1,\cdots, s_{q-1}, t_1,\cdots, t_{q-1},t_0}  = \#	\Omega_{s_0-1,s_1,\cdots, s_{q-1}, t_1,\cdots, t_{q-1}, t_0}+1  $ if and only if
		\[   \sum\limits_{\nu=0 }^{q-1} \ceil{\dfrac{qs_0-t_{\nu}}{q^2+q} +\dfrac{1}{q} \sum\limits_{\mu=1 }^{q-1} \ceil{\dfrac{s_0-s_{\mu}}{q^2+q} } } 
		\leqslant s_0  .\]
		\item
		$ \#\Omega_{s_0,s_1,\cdots, s_{q-1}, t_1,\cdots, t_{q-1}, t_0}  = \#	\Omega_{s_0,s_1,\cdots, s_{q-1}, t_1,\cdots, t_{q-1}, t_0-1}+1   $ if and only if
		\[   \sum_{\mu=0 }^{q-1} \ceil{\dfrac{qt_0-s_{\mu}}{q^2+q} +\dfrac{1}{q} \sum_{\nu=1 }^{q-1} \ceil{\dfrac{t_0-t_{\nu}}{q^2+q} } }  
		\leqslant   t_0    .\]
	\end{enumerate}
\end{lem}

\begin{proof}
	We only prove the first assertion. 
	Consider two lattice point sets $ \Omega_{s_0,s_1,\cdots, s_{q-1}, t_1,\cdots, t_{q-1},t_0} $  and $	\Omega_{s_0-1,s_1,\cdots, s_{q-1}, t_1,\cdots, t_{q-1}, t_0} $, which are given in Equation \eqref{eq:Ome_t2}. We obtain the complementary set $ \Phi $ of $ \Omega_{s_0-1,s_1,\cdots, s_{q-1}, t_1,\cdots, t_{q-1},t_0} $ in $ \Omega_{s_0,s_1,\cdots, s_{q-1}, t_1,\cdots, t_{q-1}, t_0} $ as follows:
	\begin{align*}
	\Phi:= \Big\{ \, &  (i,j_0, j_1,\cdots,j_{q-1} , k_1,\cdots,k_{q-1} )   \, \Big|\,   i= -s_0  ,  \\
	&  k_\mu = \ceil{ \dfrac{-i-s_{\mu}}{q^2+q}} \textup{ for } \mu =1, \cdots, q-1,   \\
	&j_\nu= \ceil{\dfrac{-qi-t_{\nu}}{q^2+q} +\dfrac{1}{q} \sum_{\mu=1 }^{q-1} k_\mu } \textup{ for } \nu =0, \cdots, q-1,   \\
	&i+\sum_{\nu=0}^{q-1} j_{\nu} \leqslant 0  \,\Big\}.  
	\end{align*}
	It follows immediately that the set $ \Phi $ is not empty if and only if the inequality 
	$  \sum_{\nu = 0}^{q-1} j_\nu   \leqslant  s_0    $ holds, which concludes the first assertion.
\end{proof}

The main results of this section are given below dealing with the Weierstrass semigroups and the pure gap sets. For simplicity, we write $ \mathcal{Q}_1  := Q^{(3)} $,
$  \mathcal{P}_1  :=P^{(3)} $ and $ \mathcal{Q}_{\mu+1}  := S_{0, \mu }^{(3)} $, $ \mathcal{P}_{\nu+1}  := S_{1, \nu }^{(3)} $ for $ 1\leqslant \mu, \nu \leqslant q-1 $.

\begin{thm}
	\label{lem:Weierstsemi2}
	Let $ 1\leqslant k,l \leqslant q $.
	Define 
	\begin{align*}
	S_j(l,k)= &  \sum_{\nu=1   }^{l} \ceil{\dfrac{q s_j -t_\nu}{q^2+q}+\frac{1}{q}\sum_{\mu=1 \atop \mu \neq j}^{k}\ceil{ \dfrac{s_j-s_{\mu}}{q^2+q}}+\frac{q-k}{q} \ceil{ \dfrac{s_j }{q^2+q}}}
	\\ 
	&+(q-l) \ceil{ \dfrac{s_j}{q+1}+\frac{1}{q}\sum_{\mu=1 \atop \mu \neq j}^{k}\ceil{ \dfrac{s_j-s_{\mu}}{q^2+q} } +\frac{q-k}{q} \ceil{ \dfrac{s_j }{q^2+q}}},\\
	T_i(k,l)=	&  \sum_{\mu=1  }^{k} \ceil{\dfrac{q t_i -s_\mu}{q^2+q}+\frac{1}{q}\sum_{\nu=1 \atop \nu\neq i}^{l}\ceil{ \dfrac{t_i-t_{\nu}}{q^2+q}}+\frac{q-l}{q} \ceil{ \dfrac{t_i }{q^2+q}}}
   \\ 
	&+(q-k) \ceil{ \dfrac{t_i}{q+1}+\frac{1}{q}\sum_{\nu=1 \atop \nu\neq i}^{l}\ceil{ \dfrac{t_i-t_{\nu}}{q^2+q} } +\frac{q- l}{q} \ceil{ \dfrac{t_i }{q^2+q}}}  .
	 \end{align*}The Weierstrass semigroups and the pure gap sets are given as follows. 
	\begin{enumerate}
		\item The Weierstrass semigroup $ H(  \mathcal{Q}_1  ,\cdots, \mathcal{Q}_k  \,)$ is given by
		\begin{align*}
		\Big\{\,&(s_1,\cdots,s_k)\in \mathbb{N}_0^k~\Big |
		~ \ceil{\dfrac{s_j}{q+1}+\frac{1}{q}\sum_{\mu=1 \atop \mu\neq j}^{k}\ceil{ \dfrac{s_j-s_\mu}{q^2+q}}+\frac{q-k}{q}\ceil{ \dfrac{s_j}{q^2+q}}}
		 \leqslant \dfrac{s_j }{q}
		\textup{ for all } 1\leqslant j \leqslant k \Big\}.
		\end{align*}
		\item The Weierstrass semigroup 	
		$	  H(  \mathcal{Q}_1  ,\cdots, \mathcal{Q}_k , \mathcal{P}_1 ,\cdots, \mathcal{P}_l) $ is given by
		\begin{align*}
		\Big\{\,(s_1,\cdots,s_k, t_1,\cdots,t_l)\in \mathbb{N}_0^{k+l}~\Big |\,& 
      	 S_j(l,k) \leqslant s_j   
      	 \textup{ for all } 1\leqslant j \leqslant k ,\\
      	 &   T_i(k,l) \leqslant t_i  
      	  \textup{ for all } 1\leqslant i \leqslant l        	   				
\,\,\Big\}.
		\end{align*}		
		\item The pure gap set $ G_0(  \mathcal{Q}_1  ,\cdots, \mathcal{Q}_k  \,)$ is given by
		\begin{align*}
		\Big\{\,&(s_1,\cdots,s_k)\in \mathbb{N}^k~\Big |
		~\ceil{\dfrac{s_j}{q+1}+\frac{1}{q}\sum_{ \mu=1 \atop \mu\neq j}^{k}\ceil{ \dfrac{s_j-s_{\mu}}{q^2+q}}+\frac{q-k}{q}\ceil{ \dfrac{s_j}{q^2+q}}}
		 > \dfrac{s_j }{q}
		\textup{ for all } 1\leqslant j \leqslant k \Big\}.
		\end{align*}
		\item The pure gap set 	
		$	  G_0(  \mathcal{Q}_1  ,\cdots, \mathcal{Q}_k , \mathcal{P}_1 ,\cdots, \mathcal{P}_l) $ is given by
		\begin{align*}
		\Big\{\,(s_1,\cdots,s_k, t_1,\cdots,t_l)\in \mathbb{N}^{k+l}~\Big |\,& 
      	 S_j(l,k) > s_j   
      	 \textup{ for all } 1\leqslant j \leqslant k ,\\
      	 &   T_i(k,l) > t_i 
      	 \textup{ for all } 1\leqslant i \leqslant l  
		\,\,\Big\}.
		\end{align*}	
	\end{enumerate}
\end{thm}
\begin{proof}
	 We only show the second assertion in details since the others are similarly proved. Let us consider a special divisor $  G = \sum_{\mu=1}^{q} s_{\mu} \mathcal{Q}_{\mu}  + \sum_{\nu=1}^{q}  t_{\nu} \mathcal{P}_{\nu}   $, where $ s_{k+1}=\cdots=s_q=0 $ and $ t_{l+1}=\cdots=t_q=0 $. From Proposition \ref{prop:basis2}, Lemmas \ref{lem:set_inv2}, \ref{lem:Weierstsemi} and \ref{lem:omegawith1}, we know that $ (s_1,\cdots,s_k, t_1,\cdots,t_l)\in \mathbb{N}_0^{k+l} $ is in $	  H(  \mathcal{Q}_1  ,\cdots, \mathcal{Q}_k , \mathcal{P}_1 ,\cdots, \mathcal{P}_l) $ if and only if 
		\[   \sum\limits_{\nu=1 }^{q } \ceil{\dfrac{qs_j-t_{\nu}}{q^2+q} +\dfrac{1}{q} \sum\limits_{\mu=1  \atop \mu\neq j }^{q} \ceil{\dfrac{s_j-s_{\mu}}{q^2+q} } } 
		\leqslant s_j  \]
	 for all $ 1\leqslant j \leqslant k $, and 
	\[   \sum_{\mu=1 }^{q } \ceil{\dfrac{qt_i-s_{\mu}}{q^2+q} +\dfrac{1}{q} \sum_{\nu=1 \atop \nu\neq i }^{q } \ceil{\dfrac{t_i-t_{\nu}}{q^2+q} } }  
	\leqslant   t_i    \]
	for all $ 1\leqslant i \leqslant l$. The desired conclusion then follows.
\end{proof}

\begin{cor}\label{cor:onepoint}
	With notation as before, we have the following assertions.
	\begin{enumerate}
		\item The Weierstrass semigroup $ H(  \mathcal{Q}_1   )$ is given by
		\begin{align*}
		\Big\{\alpha\in \mathbb{N}_0 ~\Big |
		~  \ceil{\dfrac{\alpha}{q+1}+\frac{q-1}{q}\ceil{ \dfrac{\alpha}{q^2+q}}}   \leqslant \dfrac{ \alpha  }{q }\Big\}.
		\end{align*}
		\item The pure gap set $ G(  \mathcal{Q}_1   )$ is given by
		\begin{align*}
		\Big\{\,(q^2+q)m+r\in \mathbb{N} ~\Big | \, 1\leqslant r < q^2+q,\,\, m \geqslant 0,\,\,
	 \ceil{\dfrac{r}{q+1}-\frac{m+1}{q} }   >  \dfrac{  r }{q } -1\,\, \Big\}.
		\end{align*}
		\item The Weierstrass semigroup $ H(  \mathcal{Q}_1 ,\mathcal{P}_1  )$ is given by
		\begin{align*}
		\Big\{(\alpha,\beta)\in \mathbb{N}_0^2 ~\Big |\,
		&\ceil{\frac{q \alpha-\beta}{q^2+q} +\frac{q-1}{q}\ceil{\frac{\alpha}{q^2+q} }}  
		+ (q-1)\ceil{\frac{  \alpha }{q+1} +\frac{q-1}{q}\ceil{\frac{\alpha}{q^2+q} }}  \leqslant \alpha,
		\\
		& \ceil{\frac{q \beta-\alpha}{q^2+q} +\frac{q-1}{q}\ceil{\frac{\beta}{q^2+q} }}  
		+ (q-1)\ceil{\frac{ \beta}{q+1} +\frac{q-1}{q}\ceil{\frac{\beta}{q^2+q} }}  \leqslant \beta
		\,\,\Big\}.
		\end{align*}	
		\item The pure gap set $ G_0(  \mathcal{Q}_1 ,\mathcal{P}_1  )$ is given by
		\begin{align*}
		\Big\{\,\big((q^2+q)m_1+r_1,&\,(q^2+q)m_2+r_2\big)\in \mathbb{N}^2 ~\Big | \, 1\leqslant r_1,r_2 < q^2+q,\,\, m_1,m_2 \geqslant 0,\,\,\\
		&
		 \ceil{\dfrac{qr_1-r_2}{q^2+q}-\frac{m_1+1}{q} }  +(q-1)\ceil{\dfrac{r_1 }{q+1}-\frac{m_1+1}{q} }  > m_2+r_1-q\,\\
		&
	\ceil{\dfrac{qr_2-r_1}{q^2+q}-\frac{m_2+1}{q} }  +(q-1)\ceil{\dfrac{r_2 }{q+1}-\frac{m_2+1}{q} }  > m_1+r_2-q\, \,\,			 \Big\}.
		\end{align*}		
	\end{enumerate}	
\end{cor}
\begin{proof}
	The first assertion and the third assertion follow immediately from Theorem \ref{lem:Weierstsemi2}. For the second assertion, we note from Theorem \ref{lem:Weierstsemi2} that   
	\begin{align*}
	G(  \mathcal{Q}_1   )=\Big\{\alpha\in \mathbb{N} ~\Big |
	~  \ceil{\dfrac{\alpha}{q+1}+\frac{q-1}{q}\ceil{ \dfrac{\alpha}{q^2+q}}}  > \dfrac{  \alpha  }{q } \Big\}.
	\end{align*}
	By taking $ \alpha= (q^2+q)m+r $ for $ m\geqslant 0 $ and $ 0\leqslant r < q^2+q $, the inequality in $ 	G(  \mathcal{Q}_1   ) $ becomes
	\begin{align}\label{eq:rm}
	 \ceil{\dfrac{r}{q+1}-\frac{m}{q}-\frac{1}{q}\ceil{ \dfrac{r}{q^2+q}}} +  \ceil{ \dfrac{r}{q^2+q}} >   \dfrac{  r  }{q }  .
	\end{align}	
	If $ r=0 $, then it follows from \eqref{eq:rm} that
	\begin{align*} 
	  \ceil{ -\frac{m}{q} }  > 0  ,
	\end{align*}
	contradicting the fact $ m\geqslant 0 $. This means that $ 1\leqslant r < q^2+q $ and so $ \ceil{ \dfrac{r}{q^2+q}}=1 $. This gives the desired assertion as required.		
	The fourth assertion is similarly proved. 
\end{proof}

Several examples for pure gaps are showed in the following.
\begin{example}
	Let $ q=5 $. From Corollary \ref{cor:onepoint} there are $ 116 $ pure gaps in the set $ G(\mathcal{Q}_1 ) $ given by
	\begin{align*}
	\left\{
	\begin{array}{ll}
	 	\,1,\,2,\,3,\,4,\,5,\,6,\,7,\,8,\,9,\,10,
	\,11,\,12,\,13,\,14,\,15,\,16,\,17,\,18,\,19,\,20,\,21,\,22,\,23,\,24,\\
	 \,26,\,27,\,28,\,29,\,31,\,32,
	\,33,\,34,\,35,\,36,\,37,\,38,\,39,\,40,\,41,\,42,\,43,\,44,\,45,\,46,\,47,\\
	 \,48,\,49,\,51,\,52,\,53,
	\,54,\,57,\,58,\,59,\,61,\,62,\,63,\,64,\,65,\,66,\,67,\,68,\,69,\,70,\,71,\,72,\\
	 \,73,\,74,\,76,\,77,
	\,78,\,79,\,82,\,83,\,84,\,88,\,89,\,91,\,92,\,93,\,94,\,95,\,96,\,97,\,98,\,99,\\ 
	\,101,\,102,\,103,
	\,104,\,107,\,108,\,109,\,113,\,114,\,119,\,121,\,122,\,123,\,124,\,127,\\ 
	\,128,\,129,\,133,\,134,
	\,139,\,152,\,153,\,154,\,158,\,159,\,164,\,183,
	\,184,\,189,\,214\, 	 
	 \end{array}	
	\right\}.
	\end{align*} 	
\end{example}
\begin{example}
	Let $ q=3 $. The $ 167 $ pure gaps are in the set $ G_0(\mathcal{Q}_1 ,Q_2 ) $ given below:
	\begin{align*}
	A\cup \widehat{A} \cup \Big\{\,  ( 1,1 ), 
	( 2,2 ), 
	( 3,3 ), 
	( 4,4 ), 
	( 5,5 ), 
	( 7,7 ), 
	( 8,8 ), 
	( 11,11 ), 
	( 14,14 ) \, \Big\},
	\end{align*}
	where 
	\begin{align*}
	A := \left\{
	\begin{array}{ll}
	 ( 1,2 ), 
	( 1,3 ), 
	( 1,4 ), 
	( 1,5 ), 
	( 1,6 ), 
	( 1,7 ), 
	( 1,8 ), 
	( 1,10 ), 
	( 1,11 ), 
	( 1,13 ), 
	( 1,14 ), 
	( 1,15 ), \\ 
	( 1,16 ), 
	( 1,17 ), 
	( 1,19 ), 
	( 1,20 ), 
	( 1,23 ), 
	( 2,3 ), 
	( 2,4 ), 
	( 2,5 ), 
	( 2,6 ), 
	( 2,7 ), 
	( 2,8 ), 
	( 2,10 ),\\  
	( 2,11 ), 
	( 2,13 ),
	( 2,14 ), 
	( 2,15 ), 
	( 2,16 ), 
	( 2,17 ), 
	( 2,19 ), 
	( 2,20 ), 
	( 2,23 ), 
	( 2,26 ), 
	( 2,29 ), \\ 
	( 3,4 ), 
	( 3,5 ), 
	( 3,6 ), 
	( 3,7 ), 
	( 3,8 ), 
	( 3,10 ), 
	( 3,11 ), 
	( 3,13 ), 
	( 3,14 ), 
	( 4,5 ), 
	( 4,6 ), 
	( 4,7 ), \\ 
	( 4,8 ), 
	( 4,10 ), 
	( 4,11 ), 
	( 4,13 ), 
	( 4,14 ), 
	( 5,6 ), 
	( 5,7 ), 
	( 5,8 ), 
	( 5,10 ), 
	( 5,11 ), 
	( 5,13 ),\\  
	( 5,14 ), 
	( 5,17 ), 
	( 5,19 ), 
	( 5,20 ), 
	( 5,23 ), 
	( 5,26 ), 
	( 7,8 ),  
	( 7,10 ), 
	( 7,11 ), 
	( 7,13 ), 
	( 7,14 ), \\ 
	( 7,17 ), 
	( 8,10 ),
	( 8,11 ), 
	( 8,13 ), 
	( 8,14 ), 
	( 8,17 ), 
	( 11,13 ), 
	( 11,14 ), 
	( 11,17 ), 
	( 14,17 )\, 
	\end{array}
	\right\},
	\end{align*}
	and the set $ \widehat{A} $ associated with $ A $ is defined by $ \widehat{A} := \{\,(j,i)\,\big|\,(i,j) \in A \, \big\} $. Additionally the $ 182  $ pure gaps are in the set $ G_0(\mathcal{Q}_1 ,\mathcal{P}_1) $ given by
	\begin{align*}
	B \cup \widehat{B} \cup \Big\{\,  ( 1,1 ), 
	( 2,2 ), 
	( 3,3 ), 
	( 4,4 ), 
	( 5,5 ), 
	( 6,6 ), 
	( 7,7 ), 
	( 10,10 ), 
	( 13,13 ), 
	( 15,15 ) \,\Big\},
	\end{align*}
	where the set $ \widehat{B} $ is associated with $ B $ and
	\begin{align*}
	B := \left\{
	\begin{array}{ll}
	( 1,2 ), 
	( 1,3 ), 
	( 1,4 ), 
	( 1,5 ), 
	( 1,6 ), 
	( 1,7 ), 
	( 1,8 ), 
	( 1,10 ), 
	( 1,11 ), 
	( 1,13 ), 
	( 1,14 ), \\
	( 1,15 ), 
	( 1,16 ), 
	( 1,17 ), 
	( 1,19 ), 
	( 1,20 ), 
	( 2,3 ), 
	( 2,4 ), 
	( 2,5 ), 
	( 2,6 ), 
	( 2,7 ), 
	( 2,8 ), \\
	( 2,10 ), 
	( 2,11 ), 
	( 2,13 ), 
	( 2,14 ), 
	( 2,15 ), 
	( 2,16 ), 
	( 2,17 ), 
	( 2,19 ), 
	( 2,20 ), 
	( 2,25 ), \\
	( 2,26 ), 
	( 2,29 ), 
	( 3,4 ), 
	( 3,5 ), 
	( 3,6 ), 
	( 3,7 ), 
	( 3,8 ), 
	( 3,10 ), 
	( 3,11 ), 
	( 3,13 ), 
	( 3,14 ), \\
	( 3,15 ), 
	( 3,16 ), 
	( 3,17 ), 
	( 3,19 ), 
	( 3,20 ), 
	( 3,25 ), 
	( 3,26 ), 
	( 4,5 ), 
	( 4,6 ), 
	( 4,7 ), 
	( 4,8 ), \\
	( 4,10 ), 
	( 4,11 ), 
	( 4,13 ), 
	( 4,14 ), 
	( 4,15 ), 
	( 4,16 ), 
	( 4,17 ), 
	( 4,19 ), 
	( 5,6 ), 
	( 5,7 ), 
	( 5,8 ), \\
	( 5,10 ), 
	( 6,7 ), 
	( 6,8 ), 
	( 6,10 ), 
	( 6,13 ), 
	( 6,14 ), 
	( 6,15 ), 
	( 6,16 ), 
	( 6,17 ), 
	( 6,19 ), 
	( 6,25 ), \\ 
	( 7,8 ), 
	( 7,10 ), 
	( 7,13 ), 
	( 7,14 ), 
	( 7,15 ), 
	( 7,16 ), 
	( 10,13 ), 
	( 13,15 ), 
	( 13,16 ), 
	( 15,16 ) \,
	\end{array}
	\right\}.
	\end{align*}
	We can verify that these results coincide with Theorem \ref{lem:Weierstsemi2} and Corollary \ref{cor:onepoint}.
\end{example}


\end{document}